\documentclass[11pt, reqno, english]{amsart}  
\usepackage[utf8]{inputenc}
\usepackage[T1]{fontenc}
\usepackage{amsmath,amsthm}
\usepackage{amsfonts,amssymb}
\usepackage{mathtools}  

\DeclarePairedDelimiter{\norm}{\lVert}{\rVert}
\DeclarePairedDelimiter{\abs}{\lvert}{\rvert}

\usepackage[colorlinks=true,urlcolor=blue,linkcolor=red,citecolor=magenta]{hyperref}
\usepackage{enumerate,paralist}
\usepackage{tikz-cd}
\usepackage[left=1in,right=1in,top=1in,bottom=1in]{geometry}
\usepackage{xcolor}
\usepackage{soul}
\usepackage{algorithm}
\usepackage{algpseudocode}
\usepackage{bbm}
\usepackage{comment}
\usepackage{appendix}
\usepackage{booktabs}
\usepackage{optidef}
\usepackage{subcaption}
\usepackage[capitalize,nameinlink]{cleveref}

\algrenewcommand\algorithmicrequire{\textbf{Input:}}
\algrenewcommand\algorithmicensure{\textbf{Output:}}

\newtheoremstyle{myremark}
{7pt}{7pt}{}{}{\bf}{.}{.5em}{}
\theoremstyle{plain}
\newtheorem{theorem}{Theorem}[section]
\newtheorem{lemma}[theorem]{Lemma}
\newtheorem{corollary}[theorem]{Corollary}
\newtheorem{proposition}[theorem]{Proposition}

\theoremstyle{definition}
\newtheorem{definition}[theorem]{Definition}

\newtheorem{problem}[theorem]{Problem}


\newcounter{parentnumber}

\theoremstyle{myremark}
\newtheorem{remark}[theorem]{Remark}
\newtheorem{example}[theorem]{Example}

\DeclareMathOperator{\BD}{BD}
\DeclareMathOperator{\PD}{PD}
\DeclareMathOperator{\inj}{inj}

\DeclareMathOperator{\maximize}{maximize}
\DeclareMathOperator{\minimize}{minimize}

\newcommand{\eps}{\varepsilon}
\newcommand{\PDset}{\mathcal{D}}
\newcommand{\PDsetextended}{\tilde{\mathcal{D}}}
\newcommand{\gradedPDset}{\bar{\mathcal{D}}}
\newcommand{\gradedPDsetextended}{\check{\mathcal{D}}}
\newcommand{\DM}{\mathcal{M}}
\newcommand{\tran}{\top}
\DeclareMathOperator{\proj}{proj}

\newcommand{\N}{\mathbb{N}}

\newcommand{\R}{\mathbb{R}}

\newcommand{\Z}{\mathbb{Z}}

\begin{document}

\title[Topological feature selection for time series data]{Topological feature selection for time series data}

\author{Peter Bubenik}
\address[PB]{Department of Mathematics, University of Florida, Gainesville, FL 32611, USA}
\email{peter.bubenik@ufl.edu}

\author{Johnathan Bush}
\address[JB$^\ast$]{Department of  Mathematics and Statistics, James Madison University, Harrisonburg, VA 22807, USA}
\email{bush3je@jmu.edu}
\address{$^\ast$Corresponding author.}

\begin{abstract}
We use tools from applied topology for feature selection on time series data.
We develop a method for scoring the variables in a multivariate time series that reflects their contributions to the topological features of the corresponding point cloud. 
Our approach produces a piecewise-linear Lipschitz gradient path in the standard geometric simplex that starts at the barycenter, which weights the variables equally, and ends at the score.
Taking the mean of stochastic gradient descent results in a mean gradient path that satisfies a strong law of large numbers and central limit theorem.
Our theory is motivated by the analysis of the neuronal activities of the nematode \emph{C.~elegans}, and our method selects an informative subset of the neurons that optimizes the coordinated dynamics.
\end{abstract}

\maketitle

\section{Introduction}\label{sec:introduction}

Many physical and biological systems exhibit periodic or quasiperiodic dynamics driven by coordinated interactions among their components. 
These dynamics are often recorded as vector-valued time series, e.g., calcium imaging of neural populations or multi-sensor physiological recordings.   
In the spirit of dimensionality reduction, we consider the problem of assigning scores to the variables of these vector-valued time series based on their contribution to the global, coordinated dynamics. 
We refer to this problem as \emph{topological feature selection for time series}. 

Broadly, feature selection is a process aiming to identify the most relevant and informative features or variables from a given data set. 
A primary goal of feature selection is to reduce dimensionality by eliminating irrelevant or redundant attributes.
In our setting, we use persistent homology to measure a feature's relevancy and perform gradient descent to identify relevant features. 
We observe that in practice our method identifies sparse combinations of variables that align with known structure in synthetic and biological data. 

Our primary motivating example is calcium imaging data from the model organism
\emph{C.~elegans}, a nematode whose nervous system comprises just $302$ neurons.
The neural activity of \emph{C.~elegans} is known to evolve cyclically near a
low-dimensional manifold~\cite{kato2015global}, and a central question is which
neurons drive this coordinated behavior.
Our methods show that the coordinated dynamics recorded across a large portion of the \emph{C.\thinspace elegans} brain are well-represented by a small subset of neurons.

In more detail, we assume there is a possibly discontinuous function $f:\R \to \R^p$, underlying our data, which is usually unknown.
Let $t_1 \in \R$, $\tau > 0$, and let $n$ be a positive integer.
Let $T = \{t_i\}_{i=1}^n$, where $t_i = t_1 + (i-1)\tau$. 
We assume that our data consists of observations $g:T \to \R^p$ given by $g_j(t_i) = f_j(t_i) + \eps_{ij}$ for some unknown error term $\eps_{ij}$.
Letting $x_i = g(t_i)$, we have the time series $X = (x_1,\ldots,x_n)$ in $\R^p$.

Now consider an element $v$ in the standard geometric $(p-1)$-dimensional simplex, $\abs{\Delta^{p-1}}$, that is, $v \in \R^p$, $v_j \geq 0$ for all $j$, and $v_1 + \cdots + v_p = 1$.
From the times series $X$ we obtain the convex-combination time series $vX$ by scaling the $j$th coordinate of $x_i$ by $v_j$ for all $i,j$.
Let $\Psi(v)$ denote the persistence diagram (Section~\ref{sec:PD}) of the Vietoris-Rips complex~\cite{Virk:book} of the sliding window embedding  (Section~\ref{sec:sliding-windows})
of $vX$.
This diagram captures the topology of the sliding window embedding of $vX$; features with large lifetimes reflect robust geometric structure in the resulting point cloud.
Let $F$ be a real-valued function on the set of persistence diagrams.
We have the following optimization problem, which we call topological feature selection for time series (\cref{problem:tfs-time-series}).
    \begin{gather*}
        \text{maximize}_v \quad (F \circ \Psi)(v)\\
        \text{subject to} \quad v \in \abs{\Delta^{p-1}}
    \end{gather*}
Let $v^*\in \abs{\Delta^{p-1}}$ be a solution of our optimization algorithm.
We interpret each component $v_j^*$ as a score of the contribution of the $j$th components of the time series to the coordinated dynamics as quantified by persistent homology.

More generally, we address the problem of topological feature selection for an arbitrary finite collection of weights on a combinatorial simplex (\cref{problem:tfs-weights}), a special case of which is given by weights induced by symmetric matrices (\cref{problem:tfs-matrices}). 
We show that as these weights vary along affine paths, the associated persistence vineyards are piecewise linear (\cref{thm:PL-w-PD,thm:PL-M-PD}). 
From our analysis of persistence vineyards, we recover the $q$-Wasserstein stability of persistence diagrams (\cref{thm:stability-p}) for $1 \leq q \leq \infty$ and also prove some related results (Section \ref{sec:stability}).

For persistence-based objective functions that are piecewise linear, e.g., total persistence, we are able to compute gradients
explicitly and efficiently (Section~\ref{sec:tfs-pl})
and gradient descent provides a convergent path in the $(p-1)$-dimensional simplex.
For a much broader class of objective functions, projected stochastic subgradient descent likewise provides such convergent piecewise linear paths (Section~\ref{sec:tfs-sgd}).
We prove that for functionals on persistence diagrams that are Lipschitz with respect to a $q$-Wasserstein distance, evaluation of these functionals on persistence diagrams obtained from our gradient path is Lipschitz (\cref{cor:stability-FV-A,cor:stability-FV-3}).

Averaging our stochastic gradient paths or our gradient paths for which the observed data has noise added, produces mean gradient paths (Section~\ref{sec:mean-gradient-path}), which satisfy a strong law of large numbers (\cref{thm:slln}) and a central limit theorem (\cref{thm:clt}).
As a special case, our assigned scores satisfy a strong law of large numbers and central limit theorem.

Our analysis of sliding window embeddings of time series and linear-combination time series 
(\Cref{sec:time-series-sliding-windows}), which uses matrices and Hadamard products to express distance matrices for sliding window embeddings of linear-combination time series, may be of independent interest.
When using the $1$ norm, the distance matrix for the sliding window embedding of a multivariate time series is the sum of the distance matrices of the sliding window embeddings of the component time series. 
Furthermore, the distance matrix for the sliding window embedding is a convolution of the distance matrices for the sliding window embedding with a window length of one.

\subsection{Related work}\label{sec:related-work}

Foundational results regarding the persistent homology of sliding window embeddings were established by Perea and Harer~\cite{perea2015sliding}.
Persistent homology has been used to analyze time series data in many settings, including gene expression~\cite{dequeant2008comparison, perea2015sw1pers}, nematode locomotion and behavior~\cite{thomas2021topological}, periodic video~\cite{tralie2016high}, climate analysis~\cite{berwald2014automatic}, wheeze detection~\cite{emrani2014persistent}, and cryptocurrency variability~\cite{gidea2020topological}, for example. 
These works commonly employ sliding window embeddings of varying subsets of data to draw conclusions about the statistics of the resulting collection of persistence diagrams. 
In this paper, we focus instead on sliding window embeddings across the entirety of a fixed data set. 

In our setting, persistence vineyards~\cite{cohen2006vines}, which track the evolution of persistence diagrams along a parameterized family of filtrations, arise naturally as we vary linear combinations of time series. 
We are indebted to the foundational work of Cohen-Steiner, Edelsbrunner and Morozov~\cite{cohen2006vines}.
The proofs of our stability theorems that follow from our study of vineyards are essentially the same as those of Skraba and Turner~\cite{Skraba:2020} and Cohen-Steiner, Edelsbrunner and Morozov~\cite{cohen2006vines}.
An alternative approach to studying time series and persistence diagrams is given by path signatures~\cite{Giusti_2023}.

Our primary motivating example is neuronal data collected from  \emph{Caenorhabditis elegans}, a transparent roundworm with a simple nervous system comprising just $302$ neurons. 
Despite its simplicity, the \emph{C.\thinspace elegans} brain is complex enough to exhibit sensory modalities including a response to heat and odor. 
In particular, neurons in \emph{C.\thinspace elegans} are known to share information by engaging in coordinated dynamic activity that evolves cyclically near a low-dimensional manifold~\cite{kato2015global}. 
Notably, Kato et al.~\cite{kato2015global} use principal component analysis to reconstruct this activity, whereas we observe low-dimensional cyclic activity through a sliding window embedding.
We remark that persistent homology has been used to study the connectivity of the \emph{C.\thinspace elegans} nervous system~\cite{helm2020growing}, as well as to construct circular coordinates on the \emph{C.\thinspace elegans} neural activity phase space~\cite{blumberg2024subsamplingaligningaveragingcircular}.

Differentiation and the optimization of persistence-based functions is a topic of active research~\cite{bruel2020topology,carriere2021optimizing,gabrielsson2020topology,  leygonie2021framework,moor2020topological,MR4712290, MR4753555,  poulenard2018topological, pmlr-v235-scoccola24a}. 
We are particularly indebted to the work of Carri{\'e}re et al.~\cite{carriere2021optimizing} for showing how the work of Davis et al.~\cite{MR4056927} on stochastic subgradient descent applies to the optimization of functions of persistent homology and provides guarantees for convergence.
Leygonie, Oudot, and Tillmann~\cite{leygonie2021framework} also develop a framework for differential calculus on the space of persistence barcodes~\cite{carlsson2004persistence, barannikov1994framed}.

\subsection{Organization} 
In Section~\ref{sec:persistent-homology}, we summarize basic results of persistent homology in the setting of weights on an augmented combinatorial simplex.
In Section~\ref{sec:vineyards}, we prove that a piecewise linear (PL) curve of weights produces a PL vineyard. 
In addition, we show that it follows that persistence diagrams are stable for various choices of metrics. 
In Section~\ref{sec:time-series-sliding-windows}, we study time series and their sliding window embeddings. 
In Section~\ref{sec:feature-selection}, we make our notion of topological feature selection for time series data precise.
In Section~\ref{sec:optimization}, we define two gradient descent algorithms for this optimization problem and prove they result in Lipschitz paths that satisfy a strong law of large numbers a central limit theorem. 
Finally, we apply our method to synthetic and biological data in Section~\ref{sec:examples}. 

\subsection{Notation}

Here we give a brief glossary of the main notation we will introduce.
$\BD_k^\mathcal{A}(K,\trianglelefteq)$ is the birth-death matching in degree $k$ obtained by applying the persistence algorithm $\mathcal{A}$ to the simplex $K$ with a linear order $\trianglelefteq$ compatible with the face poset (Section~\ref{sec:PD}).
$\PD_k^\mathcal{A}(K,\trianglelefteq,w)$
is the persistence diagram in degree $k$ obtained by applying the persistence algorithm $\mathcal{A}$ to the simplex $K$ whose simplices have scalar values compatible with the face poset given by the weight $w$ and using the linear order $\trianglelefteq$ to break ties (Section~\ref{sec:PD}).
$W$ is the space of all weights (Section~\ref{sec:weight-space}).
$\mathcal{P}_\trianglelefteq$ is the partition of $W$ induced by a linear order $\trianglelefteq$ compatible with the face poset (Section~\ref{sec:weight-space}).
$\mathcal{M}$ is a set of $m \times m$ symmetric matrices (Section~\ref{sec:symmetric-matrices}).
$\PDset$ and $\gradedPDset$ are the sets of persistence diagrams and graded persistence diagrams, respectively (Section~\ref{sec:PD}).
$\PDsetextended$ and $\gradedPDsetextended$ are the sets of extended persistence diagrams and extended graded persistence diagrams, respectively (Section~\ref{sec:stability}).

\section{Persistent homology}\label{sec:persistent-homology}

In this section, we give basic results of persistent homology and introduce notation in the setting of weights on the augmented combinatorial simplex.

\subsection{Weights on the augmented combinatorial simplex}\label{sec:prelim}

To start, we define the augmented combinatorial simplex $K$ as well as weights on $K$.

For $m\geq 1$, let $[m]$ denote the linearly ordered set $\{1 < 2 < \cdots < m\}$. 
Let $K$ denote the power set of $[m]$.
The elements of $K$ are called \emph{simplices}, and the \emph{dimension} of a simplex $\sigma\in K$, denoted $\dim(\sigma)$, is given by $|\sigma|-1$.
We call $K$ the \emph{augmented combinatorial $(m-1)$-simplex}; it is an example of an augmented abstract simplicial complex~\cite{Goerss:2009}.
We write $\hat{K}$ to denote the \emph{combinatorial $(m-1)$-simplex}, $K \setminus \varnothing$. 
$\hat{K}$ is a simplicial complex and $K$ is an augmented simplicial complex.

The set $K$ has a partial order $\subseteq$ given by subset inclusion.
The poset $(K,\subseteq)$ is called the \emph{face poset} of $K$. 
The Hasse diagram of $(K,\subseteq)$ has directed edges given by ordered pairs of simplices $(\sigma,\tau)$ where $\sigma \subseteq \tau$ and $\dim(\tau) = \dim(\sigma) + 1$.
See \cref{fig:hasse-diagram}.
A \emph{linear extension} of $(K, \subseteq)$ is a total order $\trianglelefteq$ on $K$ such that $\sigma\subseteq \tau$ implies $\sigma\trianglelefteq\tau$. 
See \cref{fig:hasse-diagram}.

A \emph{weight} on $K$ is an order preserving map $w\colon (K,\subseteq)\to (\R, \leq )$. 
See \cref{fig:hasse-diagram}.
A weight $w$ on $K$ induces 
a total preorder $\preceq_w$ on $K$ given by $\sigma \preceq_w \tau$ if $w(\sigma) \leq w(\tau)$. 
Any linear extension $\trianglelefteq$ of the subset order on $K$ can be used to refine $\preceq_w$ to  a total order 
$\leq_{(w,\trianglelefteq)}$ on $K$
given by $\sigma \leq_{(w,\trianglelefteq)} \tau$ if either $w(\sigma) < w(\tau)$ or $w(\sigma) = w(\tau)$ and $\sigma \trianglelefteq \tau$,
i.e. we use $\trianglelefteq$ to break ties.
See \cref{fig:hasse-diagram}.

\begin{figure}[htb!]
    \centering
    (a)
    \begin{tikzcd}[column sep=tiny]
        & 123\\
        12 \ar[ur] & 13 \ar[u] & 23 \ar[ul] \\
        1 \ar[u] \ar[ur] & 2 \ar[ul] \ar[ur] & 3 \ar[u] \ar[ul] \\
        & \emptyset \ar[ul] \ar[u] \ar[ur]
    \end{tikzcd} 
    (b)
    \begin{tikzcd}[column sep=tiny, row sep=tiny]
        & 123\\
        & & 23 \ar[ul]\\
        & 13 \ar[ur]\\
        12 \ar[ur] \\
        & & 3 \ar[ull]\\
        & 2 \ar[ur]\\
        1 \ar[ur]\\
        & \emptyset \ar[ul]
    \end{tikzcd} 
    (c)
    \begin{tikzcd}[column sep=tiny, row sep=small]
        & & & w\\
        12 & 123 & 23 & 4\\
        & 13 & & 3\\ 
        1 
        & 2 & & 2\\
        & & 3 & 1\\ 
        & \emptyset & & 0
    \end{tikzcd} 
    (d)
    \begin{tikzcd}[column sep=tiny, row sep=tiny]
        & 123\\
        & & 23 \ar[ul]\\
        12 \ar[urr]\\
        & 13 \ar[ul] \\
        & 2 \ar[u]\\
        1 \ar[ur]\\
        & & 3 \ar[ull]\\
        & \emptyset \ar[ur]
    \end{tikzcd}

    \caption{Orders and weights for the augmented combinatorial $2$-simplex. 
    (a) The Hasse diagram of the face poset $\subseteq$.
    (b) The Hasse diagram of a linear extension $\trianglelefteq$ of $\subseteq$.
    (c) The weight $w$ given by $w(\emptyset) = 0$, $w(3) = 1$, $w(1) = w(2) = 2$, $w(13) = 3$, and $w(12) = w(23) = w(123) = 4$.
    (d) The Hasse diagram of the total order $\leq_{(w,\trianglelefteq)}$ refining $\preceq_w$ using $\trianglelefteq$.}
    \label{fig:hasse-diagram}
\end{figure}

\begin{lemma} \label{lem:weight-order}
    Let $w$ be a weight on $K$, and let $\trianglelefteq$ be a linear extension of the subset order on $K$. Let $\sigma,\tau \in K$.
    Then $\sigma \subseteq \tau$ implies that $\sigma \leq_{(w,\trianglelefteq)} \tau$ 
    and $\sigma \leq_{(w,\trianglelefteq)} \tau$ implies that $\sigma \preceq_w \tau$.
    In particular, 
    $\leq_{(w, \trianglelefteq)}$ is a linear extension of the subset order. 
\end{lemma}

\subsubsection{Weights induced by symmetric matrices} \label{sec:symmetric-matrices}

An important special case arises from symmetric matrices, such as distance matrices.

Let $K$ be the augmented combinatorial $(m-1)$-simplex. 
Consider a symmetric matrix $M=(M_{ij})_{1\leq i,j\leq m}$ 
such that $M_{i i} \leq M_{ij}$ for all $i,j$. 
This defines a weight $w_M$ on $K$ as follows. 
For each $\sigma\in K$, define 
\begin{equation} \label{eq:wM}
    w_M(\sigma)=\begin{cases} \max\limits_{i, j\in \sigma} \, M_{ij} & \text{ if }  \sigma\neq \varnothing \\
\min\limits_{1 \leq i \leq m} \, M_{ii} 
& \text{ if }  \sigma = \varnothing
\end{cases}
\end{equation}
Note that
$\sigma\subseteq \tau$ implies $w_M(\sigma)\leq w_M(\tau)$. 

\begin{remark}
    A symmetric $m \times m$ matrix $M$ with $M_{ii} \leq M_{ij}$ for all $i,j$ may be viewed as a vertex and edge weighted complete graph $(K_m,w)$ for which the weights induce a filtration of the graph.
\end{remark}

\begin{example}
\textnormal{
Consider a finite metric space $(X,d)$ with totally ordered elements $\{x_1< x_2<\cdots <x_m\}$.
Then we have a corresponding distance matrix $(d(x_i,x_j))_{1 \leq i,j \leq m}$ 
and corresponding weight $w_d$.}
\end{example}

\begin{remark}
    Given any $m \times m$ matrix $M$, equation \eqref{eq:wM} defines a weight on $K$. However, this weight equals the weight of the matrix $M'$ given by $M'_{ij} = \max \{ M_{ii}, M_{jj}, M_{ij}, M_{ji}\}$.
    The matrix $M'$ is symmetric and satisfies $M'_{ii} \leq M'_{ij}$ for all $i,j$.
\end{remark}

\subsection{Persistence diagrams}\label{sec:PD} 

We now discuss persistence diagrams arising from a weight on the augmented combinatorial simplex.

Let $K$ denote the augmented combinatorial $(m-1)$-simplex.
The $f$-vector of $K$ is given by $(f_{-1},f_0,\ldots,f_{m-1})$, where $f_k = \binom{m}{k+1}$ equals the number of $k$-simplices in $K$.
For $-1 \leq k \leq m-1$, let $f_k^+ = \sum_{j=-1}^k (-1)^{k-j} f_j$.
Since the combinatorial $(m-1)$-simplex $\hat{K}$ is contractible, the reduced simplicial homology of $\hat{K}$ is zero.
It follows that for $K^{(k)}$, the $k$-skeleton of $K$, we have that the $k$-th reduced Betti number is given by $\overline{\beta}_k(K^{(k)}) = f_k^+$.

For what follows, let $\leq$ be any linear extension of the subset order of $K$. 
There is a corresponding square matrix $D$ with entries in $\Z/2\Z$ whose rows and columns are indexed by the simplices of $K$, ordered by $\leq$, such that $D_{\sigma \tau}$ equals $1$ if $\sigma \subseteq \tau$ and $\dim \tau = \dim \sigma + 1$ and equals $0$ otherwise.
$D$ is called the \emph{boundary matrix} of $(K,\leq)$.
A square matrix $R = (R_{ij})$ is called \emph{reduced} if each nonzero column $R_{\bullet j}$ has a unique lowest nonzero entry.
A reduced matrix produces a matching of its rows and columns given by matching nonzero columns with the rows corresponding to their lowest nonzero entry.
A \emph{persistence algorithm} is a row reduction algorithm consisting of column addition operations that add a column of a matrix to a column to the right, that, taking a boundary matrix as input, produce a reduced matrix~\cite{cohen2006vines}. 

Let $\mathcal{A}$ be a persistence algorithm.
Applying $\mathcal{A}$ to the boundary matrix of $(K,\leq)$ produces a matching of the complete graph on
the simplices of $K$ consisting of pairs $(\sigma,\tau)$ with $\dim \tau = \dim \sigma +1$,
which we call a \emph{birth-death matching}
and which we denote $\BD^\mathcal{A}(K,\leq)$.
Elements $(\sigma,\tau)$ of $\BD^\mathcal{A}(K,\leq)$ are called \emph{birth-death pairs}, and $\sigma$ and $\tau$ are called \emph{birth} and \emph{death} simplices, respectively.
Let $\BD_k^\mathcal{A}(K,\leq) = \{(\sigma,\tau) \in \BD^\mathcal{A}(K,\leq)\ | \ \dim(\sigma)=k\}$. 
See \cref{ex:bd}.

\begin{figure}[htb!]
    \centering
    (a)
    \begin{tikzcd}[column sep=tiny, row sep=tiny]
        & 123\\
        12 \ar[ur] \\
        & 13 \ar[ul]\\
        & & 23 \ar[ul]\\
        1 \ar[urr]\\
        & 2 \ar[ul]\\
        & & 3 \ar[ul]\\
        & \emptyset \ar[ur]
    \end{tikzcd} 
    (b)
    \begin{tikzcd}[column sep=tiny, row sep=tiny]
        & 123\\
        12 \ar[ur]\\
        & & 23 \ar[ull]\\
        & 13 \ar[ur] \\
        1 \ar[ur]\\
        & 2 \ar[ul]\\
        & & 3 \ar[ul]\\
        & \emptyset \ar[ur]
    \end{tikzcd}
    (c)
    \begin{tikzpicture}[scale=1]
        \draw[blue, thick] (0,0) -- (1,0) -- (0.5,0.866) -- cycle;
        \filldraw[blue] (0,0) circle (1.5pt) node[black,anchor=east]{$1$};
        \filldraw[blue] (1,0) circle (1.5pt) node[black,anchor=west]{$2$};
        \filldraw[blue] (0.5,0.866) circle (1.5pt) node[black,anchor=south]{$3$};
        \node [align=center,color=black] at (0.5,-0.2) {$12$};
        \node [align=center,color=black] at (0,0.5) {$13$};
        \node [align=center,color=black] at (1,0.5) {$23$};
        \node [align=center,color=black] at (0.5,0.25) {$123$};
    \end{tikzpicture}
    (d)
    \begin{tikzpicture}[scale=1]
        \draw[blue, thick] (0,0) -- (1,0) -- (0.5,0.866) -- cycle;
        \filldraw[blue] (0,0) circle (1.5pt) node[black,anchor=east]{$2$};
        \filldraw[blue] (1,0) circle (1.5pt) node[black,anchor=west]{$3$};
        \filldraw[blue] (0.5,0.866) circle (1.5pt) node[black,anchor=south]{$1$};
        \node [align=center,color=black] at (0.5,-0.2) {$5$};
        \node [align=center,color=black] at (0.1,0.5) {$4$};
        \node [align=center,color=black] at (0.9,0.5) {$6$};
        \node [align=center,color=black] at (0.5,0.2887) {$7$};
    \end{tikzpicture}
    (e)
    \begin{tikzpicture}[scale=1]
        \draw[blue, thick] (0,0) -- (1,0) -- (0.5,0.866) -- cycle;
        \filldraw[blue] (0,0) circle (1.5pt) node[black,anchor=east]{$3$};
        \filldraw[blue] (1,0) circle (1.5pt) node[black,anchor=west]{$2$};
        \filldraw[blue] (0.5,0.866) circle (1.5pt) node[black,anchor=south]{$1$};
        \node [align=center,color=black] at (0.5,-0.2) {$6$};
        \node [align=center,color=black] at (0.1,0.5) {$5$};
        \node [align=center,color=black] at (0.9,0.5) {$4$};
        \node [align=center,color=black] at (0.5,0.2887) {$7$};
    \end{tikzpicture}
    \caption{Orders and weights for the augmented combinatorial $2$-simplex, a continuation of \cref{fig:hasse-diagram}. 
    (a) The Hasse diagram of a linear extension $\trianglelefteq'$ of $\subseteq$.
    (b) The Hasse diagram of the total order $\leq_{(w,\trianglelefteq')}$ refining $\preceq_w$ using $\trianglelefteq'$.
    (c) The combinatorial $2$-simplex.
    (d) The total order given by $\leq_{(w,\trianglelefteq)}$.
    (e) The total order given by $\leq_{(w,\trianglelefteq')}$.
    }
    \label{fig:weighted-triangles}
\end{figure}

\begin{example} \label{ex:bd}
For $K$ the augmented combinatorial $2$-simplex, the birth-death matching corresponding to the total orders $\leq_{(w,\trianglelefteq)}$ and $\leq_{(w,\trianglelefteq')}$ in \cref{fig:hasse-diagram,fig:weighted-triangles}.

\[
\begin{array}{c|c|c|c}
k & -1 & 0 & 1 \\  \hline
\BD_k^\mathcal{A}(K,\leq_{(w,\trianglelefteq)})
&
(\emptyset,3)
&
(1, 13), (2, 12)
&
(23, 123)
\\ \hline
\BD_k^\mathcal{A}(K,\leq_{(w,\trianglelefteq')})
&
(\emptyset,3)
&
(1, 13), (2, 23)
&
(12, 123)
\end{array}
\]
\end{example}

\begin{lemma}[{\cite[Pairing Uniqueness Lemma]{cohen2006vines}}] \label{lem:ph}
    The birth-death matching $\BD^\mathcal{A}(K,\trianglelefteq)$ is independent of the choice of persistence algorithm $\mathcal{A}$.
\end{lemma}

Recall that a weight $w:(K,\subseteq) \to (\R,\leq)$ induces a total preorder $\preceq_w$ on $K$, and
we may use $\trianglelefteq$ to refine this preorder to a linear extension $\leq_{(w, \trianglelefteq)}$ of the subset order.
Applying a persistence algorithm to this order,
we obtain a birth-death matching $\BD_k^\mathcal{A}\left(K,\leq_{(w,\trianglelefteq)}\right)$. 
Note that $\leq_{(w,\trianglelefteq)}$ depends on both $w$ and $\trianglelefteq$.
Let $\R^2_\leq = \{(x,y) \in \R^2~|~x \leq y\}$.
For $-1 \leq k \leq m-1$, the \emph{degree-$k$ persistence diagram} of $(K,\trianglelefteq,w)$ is defined to be
the formal sum (i.e.\ multiset) 
\begin{equation} \label{eq:pd}
    \PD_k^\mathcal{A}(K,\trianglelefteq,w) = \sum_{(\sigma,\tau) \in \BD_k^\mathcal{A}\left(K,\leq_{(w,\trianglelefteq)}\right)} (w(\sigma),w(\tau)).
\end{equation}

See \cref{ex:pd}.
The number of terms in the sum~\eqref{eq:pd} is called the \emph{cardinality} of the persistence diagram.

\begin{example} \label{ex:pd}

For $K$ the augmented combinatorial $2$-simplex, the degree-$k$ persistence diagrams corresponding to $w$ and to the total orders $\trianglelefteq$ and $\trianglelefteq'$ in \cref{fig:hasse-diagram,fig:weighted-triangles}.
\[
\begin{array}{c|c|c|c}
k & -1 & 0 & 1 \\  \hline
\PD_k^\mathcal{A}(K,\trianglelefteq,w)
&
(0,1)
&
(2,3) + (2,4)
&
(4,4)
\\ \hline
\PD_k^\mathcal{A}(K,\trianglelefteq',w)
&
(0,1)
&
(2,3) + (2,4)
&
(4,4)
\end{array}
\]
\end{example}

The formal sum $\PD_k^\mathcal{A}(K,\trianglelefteq,w)$ is an element of the free commutative monoid (the set of formal sums together with the operation given by addition) on the set $\R^2_\leq$, which we denote $D(\R^2_\leq)$.
We refer to $w(\sigma)$ as the \emph{birth}, $w(\tau)$ as the \emph{death}, and $w(\tau)-w(\sigma)$ as the \emph{lifetime} or \emph{persistence} of the birth-death pair $(\sigma,\tau)$.

For $a \in \R$, $K$ has a subcomplex $K_a$ consisting of those simplices $\sigma$ of $K$ with $w(\sigma) \leq a$.
For $a \leq b$ and $k \in \Z$, we have $K_a \subseteq K_b$, 
which induces a map on reduced simplicial homology in degree $k$ with coefficients in $\Z/2\Z$.
Let $\overline{\beta}_k^{ab}(K,w)$ denote the rank of this linear map, called the \emph{reduced persistent Betti number}.

The following is a central result of persistent homology.
For $a \leq b$, let $Q_{ab} = \{(x,y) \in \R^2 \ | \ x \leq a \leq b \leq y\}$.
We may interpret elements of $D(\R^2_{\leq})$ as integer-valued functions on $\R^2_{\leq}$ with finite support.
Furthermore, they may be viewed as discrete measures on $\R^2_{\leq}$.
Then $\PD_k^\mathcal{A}(K,\trianglelefteq,w) (Q_{ab})= \sum_{(x,y) \in Q_{ab}} \PD_k^\mathcal{A}(K,\trianglelefteq,w)(x,y)$,
where this sum is finite by definition.

\begin{lemma}[\cite{elz:tPaS}]
    $\PD_k^\mathcal{A}(K,\trianglelefteq,w) (Q_{ab}) = \overline{\beta}_k^{ab}(K,w)$ for any $a\leq b$. 
\end{lemma}

We have the following corollary.

\begin{corollary}
     \label{thm:ph}
    The persistence diagram
    $\PD_k^\mathcal{A}(K,\trianglelefteq,w)$ 
    is independent of both the total order $\trianglelefteq$ and the choice of persistence algorithm $\mathcal{A}$.
    We henceforth denote it by $\PD_k(K,w)$.
\end{corollary}

\begin{example}
\textnormal{
    Consider a finite ordered metric space $(X,d)$ of cardinality $m$. 
Then the degree-$k$ \emph{Vietoris--Rips} persistence diagram of $X$,
$\PD_k(K,w_d)$, is independent of the order on the metric space}.
\end{example}

Since $\hat{K}$ is contractible, its reduced simplicial homology is zero.
Therefore applying a persistence algorithm to the boundary matrix of $(K,\trianglelefteq)$ produces a perfect matching.
We have the following.

\begin{lemma} \label{lem:cardinality-BDk}
    The cardinality of $\BD_k^\mathcal{A}(K,\trianglelefteq)$ is given by $f_k^+ = \sum_{j=-1}^k (-1)^{k-j} f_j$, which is independent of both the total order $\trianglelefteq$ and the choice of persistence algorithm $\mathcal{A}$.
\end{lemma}

We call elements of $D(\R^2_\leq)$ \emph{persistence diagrams} and 
let $\PDset$ denote the set of persistence diagrams.
It has a family of metrics called \emph{Wasserstein distances}, $W_q$, depending on $1 \leq q \leq \infty$ and a choice of metric on $\R^2_\leq$~\cite{csehm:lipschitz,cohen2005stability}, which we now define.
Following Skraba and Turner~\cite{Skraba:2020}, 
for the $q$-Wasserstein distance we will use the metric on $\R^2_\leq$ induced by the $q$-norm on $\R^2$.

Consider persistence diagrams $\sum_{i=1}^m (b_i,d_i)$ and $\sum_{j=1}^n (b'_j,d'_j)$.
Given $A \subseteq \{1,\ldots,m\}$
and $f:A \to \{1,\ldots,n\}$,
let $B = \{1,\ldots,m\} \setminus A$ and 
let $C = \{1,\ldots,n\} \setminus f(A)$.
Let
\begin{multline} \label{eq:Wasserstein-distance}
    W_q\left(\sum_{i=1}^m (b_i,d_i), \sum_{j=1}^n (b'_j, d'_j)\right) = \min
    \Bigg\lVert \Bigg( 
    \left( \norm{ (b_i,d_i) - (b'_{f(i)},d'_{f(i)}) }_q \right)_{i \in A},\\
    \left( \norm{ (b_i,d_i) - \bigg(\frac{b_i+d_i}{2},\frac{b_i+d_i}{2}\bigg) }_q \right)_{i \in B},
    \left( \norm{ (b'_j,d'_j) - \bigg(\frac{b'_j+d'_j}{2},\frac{b'_j+d'_j}{2}\bigg) }_q \right)_{j \in C}
    \Bigg) \Bigg\rVert_q,
\end{multline}
where the minimum is taken over all
$A \subseteq \{1,\ldots,m\}$ and all injective maps
$f:A \to \{1,\ldots,n\}$.

It is sometimes convenient to combine the persistence diagrams $\PD_k(K,w)$ for all $k$.
Define $\PD(K,w) \in D(\Z \times \R^2_\leq)$ by $\PD(K,w)(k,x,y) = \PD_k(K,w)(x,y)$,
where $\PD_k(K,w)=0$ for $k < -1$ and $k > m-1$.
We call elements of $D(\Z \times \R^2_\leq)$ \emph{graded persistence diagrams} and let $\gradedPDset$ denote the set of graded persistence diagrams.
For $1 \leq q \leq \infty$, we will extend the $q$-Wasserstein distance from $\PDset$ to $\gradedPDset$
by taking the $q$-norm of the $q$-Wasserstein distances for each $k$ (which are only possibly-nonzero for $-1 \leq k \leq m-1$).

\section{PL vineyards}
\label{sec:vineyards}

In this section we prove that the persistence diagrams of a PL curve of weights has the structure of what we call a PL vineyard. 

\subsection{The space of weights} \label{sec:weight-space}

To start, we consider the structure of the space of weights on the augmented combinatorial simplex.

Let $K$ denote the augmented combinatorial $(m-1)$-simplex.
Recall that a weight on $K$ is an order preserving map $w: (K,\subseteq) \to (\R,\leq)$ and that $\abs{K} = 2^m$.
Let $W\subset \R^{2^m}$ denote the set of all weights on $K$, which we call the \emph{weight space} of $K$.
Then $W$ is a $2^m$-dimensional polyhedron (a finite intersection of closed
half-spaces) given by the intersection of the closed half-spaces given by $w(\sigma) \leq w(\tau)$, where $\sigma\subseteq \tau \in K$.

Let $W_{\inj}$ denote the subset of $W$ consisting of weights that are injective, which we call the \emph{injective weight space}.
Then any $w \in W_{\inj}$ determines a total order $\leq_w$ on $K$ given by $\sigma \leq_w \tau$ if $w(\sigma) \leq w(\tau)$.
Say that $w,w' \in W_{\inj}$ are equivalent if they induce the same total order on $K$.
This equivalence relation determines a partition $\mathcal{P}_{\inj}$ of $W_{\inj}$, which we call the \emph{order partition} of $W_{\inj}$.
Note that the elements of $\mathcal{P}_{\inj}$ correspond to linear extensions of the subset order on $K$.

We wish to extend this partition of $W_{\inj}$ to a partition of $W$.
Note that if $w \in W \setminus W_{\inj}$, then 
$w$ lies on one or more of the hyperplanes in $\R^{2^m}$ given by 
$w(\sigma) = w(\tau)$, where $\sigma, \tau \in K$.
Also, any $w \in W$ determines a total preorder $\preceq_w$ on $K$ given by $\sigma \preceq_w \tau$ if $w(\sigma) \leq w(\tau)$.
Fix a linear extension $\trianglelefteq$ of the subset order of $K$, and recall that we may use $\trianglelefteq$ to refine the total preorder $\preceq_w$ to a total order $\leq_{(w,\trianglelefteq)}$ on $K$  
given by $\sigma \leq_{(w,\trianglelefteq)} \tau$ if $w(\sigma) < w(\tau)$ or $w(\sigma) = w(\tau)$ and $\sigma \trianglelefteq \tau$. 
Note that for $w \in W_{\inj}$ this total order does not depend on $\trianglelefteq$ and coincides with the previously defined order $\leq_w$.
Say that $w,w' \in W$ are equivalent if they induce the same total order $\leq_{(w,\trianglelefteq)}$ on $K$.
This relation gives a partition of $W$ that depends on $\trianglelefteq$, which we denote $\mathcal{P}_{\trianglelefteq}$ and which we call the \emph{order partition of $W$ induced by $\trianglelefteq $}.
Note that $\mathcal{P}_{\trianglelefteq}$ extends the previously defined partition $\mathcal{P}_{\inj}$, which did not depend on $\trianglelefteq$.
We will refer to an element of a partition as a $\emph{region}$.

\begin{lemma} \label{lem:regions}
    Let $\trianglelefteq$ be a linear extension of the subset order of $K$.  
    The regions of $\mathcal{P}_{\trianglelefteq}$ are $2^m$-dimensional polyhedrons bounded by hyperplanes given by $w(\sigma)=w(\tau)$, where $\sigma, \tau \in K$.
    The interior of a region of $\mathcal{P}_{\trianglelefteq}$ is a region of $\mathcal{P}_{\inj}$.
\end{lemma}

\begin{lemma} \label{lem:region-closed}
     Let $\trianglelefteq$ be a linear extension of the subset order of $K$.  
    Each of the regions of $\mathcal{P}_{\trianglelefteq}$ corresponds to a linear extension of the subset order of $K$.
    The region of $\mathcal{P}_{\trianglelefteq}$ corresponding to $\trianglelefteq$ is closed.
\end{lemma}

\begin{lemma} \label{lem:region-containment}
    Let $\trianglelefteq$ and $\leq$ be linear extensions of the subset order of $K$. 
    The regions of $\mathcal{P}_{\trianglelefteq}$ and the regions of $\mathcal{P}_{\leq}$ are subsets of $W$.
    The region of $\mathcal{P}_{\trianglelefteq}$ corresponding to $\leq$ is contained in the region of $\mathcal{P}_{\leq}$ corresponding to $\leq$.
\end{lemma}

\subsection{Persistence diagrams of a PL curve of weights} \label{sec:pl-weights-vineyards}

We now study persistence diagrams arising from a PL curve of weights.

Let $a\leq b\in \R$. A \emph{PL vine} is a piecewise linear (i.e.\ piecewise affine) continuous function $v:[a,b] \to \R^2_\leq$.
A \emph{PL vineyard} is a function $V:[a,b] \to \PDset$ such that there exists $n \geq 0$ and PL vines $v_1,\ldots,v_n: [a,b] \to \R^2_\leq$ for which $V = \sum_{j=1}^n v_j$.
Note that since the vines in a vineyard may intersect, the decomposition of a vineyard into vines need not be unique.

\begin{figure}[ht]
\centering

\begin{subfigure}{0.25\textwidth}
  \centering
  \includegraphics[width=\textwidth]{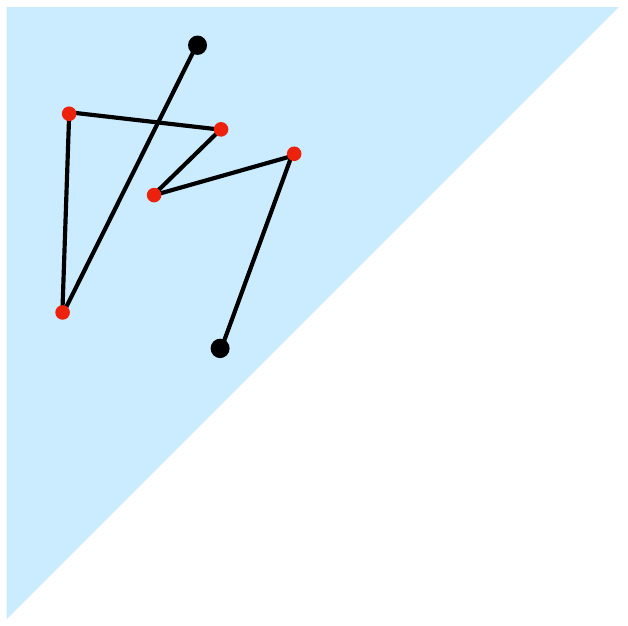}
  \label{fig:linear}
\end{subfigure}
\hspace{1cm}
\begin{subfigure}{0.28\textwidth}
  \centering
  \includegraphics[width=\textwidth]{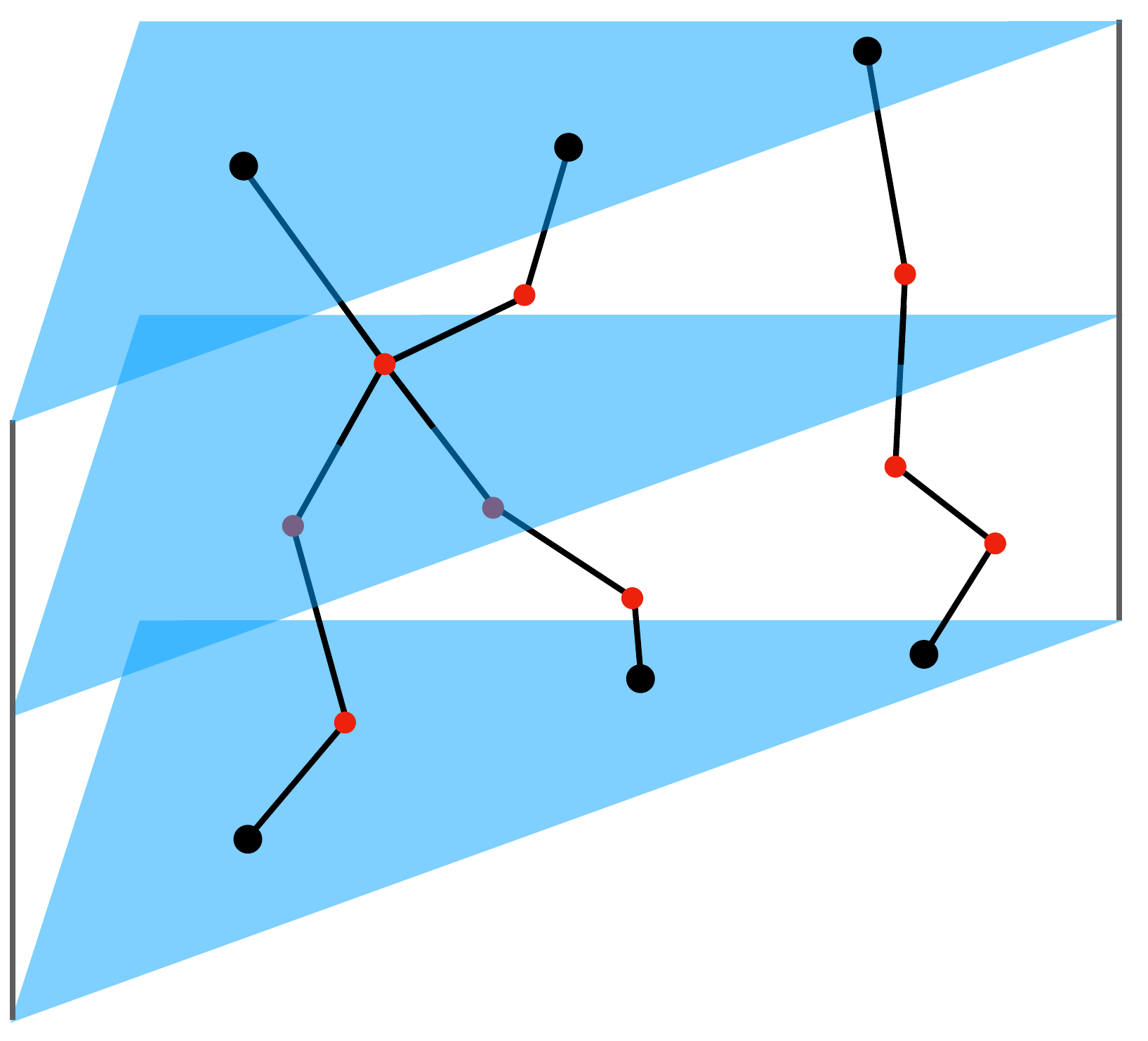}
  \label{fig:vineyard}
\end{subfigure}

\caption{Left: The image of a PL vine is a PL path in the half-plane $\R^2_\leq$. Right: A visualization of a PL vineyard whose horizontal slices are persistence diagrams, where the vertical axis represents the interval $[a,b]$. In both cases, red dots indicate where the birth-death pairing changes.}
\label{fig:combined}
\end{figure}

Let $K$ denote the augmented combinatorial $(m-1)$ complex.
Let $W\subset \R^{2^m}$ denote the weight space of $K$. 
Let $\trianglelefteq$ denote a linear extension of the subset order of $K$ and choose a persistence algorithm $\mathcal{A}$.
By \cref{lem:regions}, the total order $\trianglelefteq$ induces a partition $\mathcal{P}_{\trianglelefteq}$ of $W$ into finitely many polyhedrons.

\begin{lemma} \label{lem:pl-partition}
    Let $w:[a,b] \to W$ be a PL curve. 
    Then there is a finite partition of $[a,b]$ into intervals such that the restriction of $w$ to each of these intervals has image contained in a single region of $\mathcal{P}_{\trianglelefteq}$.    
\end{lemma}

\begin{proof}
    Since $[a,b]$ is compact, $w$ is PL, and $\mathcal{P}_{\trianglelefteq}$ is a partition consisting of finitely many polyhedrons, the result follows.
\end{proof}

\begin{proposition} \label{prop:pl-main}
    Let $w:[a,b] \to W$ be a PL curve. Let $-1 \leq k \leq m-1$.
    Let $V:[a,b] \to \PDset$ be given by $V(t) = \PD_k(K,w_t)$.
    If $w((a,b))$ lies in a single region of $\mathcal{P}_{\trianglelefteq}$ then there exist $k$-simplices $\sigma_1,\ldots,\sigma_{f_k^+}$ and $(k+1)$-simplices $\tau_1,\ldots,\tau_{f_k^+}$ of $K$ such that for $t \in [a,b]$,
    \begin{equation} \label{eq:pd-single-region}
        \PD_k(K,w_t) = \sum_{i=1}^{f_k^+} (w_t(\sigma_i),w_t(\tau_i)).
    \end{equation}
\end{proposition}

\begin{proof}
     By \cref{lem:cardinality-BDk}, for all $t$, the cardinality of $V(t)$ is $f_k^+ = \sum_{j=1}^k (-1)^{k-j} \binom{m}{j+1}$.   
    By assumption, $w((a,b))$ lies in a region of $\mathcal{P}_{\trianglelefteq}$, which corresponds to some linear extension $\leq$ on $K$ of the subset order.
    By \cref{lem:region-containment}, $w((a,b))$ is contained in the region of $\mathcal{P}_{\leq}$ corresponding to $\leq$.
    By \cref{lem:region-closed}, this region is closed.
    Since $w$ is continuous, $w([a,b])$ lies in this region.
    Thus for all $t$, $\BD_k^\mathcal{A}(K,\leq_{(w_t,\trianglelefteq)}) = \BD_k^\mathcal{A}(K,\leq)$ and
    the desired result follows from \cref{thm:ph} and \eqref{eq:pd}.
\end{proof}

\begin{theorem} \label{thm:PL-w-PD}
    Let $w:[a,b] \to W$ be a PL curve. Let $-1 \leq k \leq m-1$.
    Then $V:[a,b] \to \PDset$ given by $V(t) = \PD_k(K,w_t)$ is a PL vineyard.
\end{theorem}

\begin{proof}
 By \cref{lem:pl-partition}, there is a finite partition of $[a,b]$ into intervals such that the restriction of $w$ to each of these intervals has image contained in a single region of $\mathcal{P}_\trianglelefteq$.
    In particular, for each of these intervals, the image of its interior lies in a single region of $\mathcal{P}_\trianglelefteq$.
    By \cref{prop:pl-main}, the restriction of $V$ to the closure of each of these intervals is a PL vineyard.
    Since the partition of finite, $V$ is a PL vineyard.
\end{proof}

It will follow from work below (\cref{rem:extend-L-to-K}) that the above results extend to finite simplicial complexes.

We now specialize to the setting of weights induced by symmetric matrices. Let $\DM$ denote the set of $m \times m$ symmetric matrices considered in Section~\ref{sec:symmetric-matrices}.
That is, $\DM = \{M_{ij} \subseteq \R^{m^2} \ | \ \forall i,j, \  M_{ij} = M_{ji}, M_{ii} \leq M_{ij}\}$.

\begin{lemma} \label{lem:PL-M-wM}
    Let $M_\bullet: [a,b]  \to \DM$ be a PL curve. 
    Then $w_{M_\bullet}: [a,b] \to W$ is a PL curve.
\end{lemma}

\begin{proof}
    The entries of the matrices in the image of $M_\bullet$ are PL functions $[a,b]\to \R$, and the weight of any nonempty simplex is given by the maximum of a finite set of these functions by~\eqref{eq:wM}. 
    Hence, $w_{M_\bullet}$ is a PL curve in $W$. 
\end{proof}

Combining \cref{lem:PL-M-wM} and \cref{thm:PL-w-PD}, we have the following.

\begin{theorem} \label{thm:PL-M-PD}
    Let $M_\bullet: [a,b] \to \DM$ be a PL curve. 
    Let $-1 \leq k \leq m-1$.
    Then $V:[a,b] \to \PD$ given by $V(t) = \PD_k(K,w_{M_t})$ is a PL vineyard.
\end{theorem}

\subsection{Stability of persistence diagrams}
\label{sec:stability}

Numerous stability theorems follow easily from our results on PL vineyards.
We give a few examples.

Let $K$ denote the augmented combinatorial $(m-1)$-simplex,
and let $\trianglelefteq$ denote a linear extension of the subset order of $K$.

\begin{proposition} \label{prop:stability-main}
    Let $w$ and $w'$ be weights in a region of $\mathcal{P}_{\trianglelefteq}$ corresponding to a linear extension $\leq$ of the subset order on $K$.
    Let $1 \leq q \leq \infty$.
    Then $W_q(\PD(K,w),\PD(K,w')) \leq \norm{w-w'}_q$.
\end{proposition}

\begin{proof}
    Let $\mathcal{A}$ be a persistence algorithm.
    This persistence algorithm produces a perfect matching of the simplices of $K$.
    By assumption, $\BD^\mathcal{A}(K,\leq_{(w,\trianglelefteq)}) = \BD^\mathcal{A}(K,\leq) = \BD^\mathcal{A}(K,\leq_{(w',\trianglelefteq)})$.
    Let $\BD^\mathcal{A}(K,\leq) = \left\{(\sigma_i,\tau_i)\right\}_i$.
Then $\PD(K,w) = \sum_i (w(\sigma_i),w(\tau_i))$ and
    $\PD(K,w') = \sum_i (w'(\sigma_i),w'(\tau_i))$.
    Using \eqref{eq:Wasserstein-distance} and the identity map (for which $B$ and $C$ are empty)
    we have
    \begin{multline*}
        W_q(\PD(K,w),\PD(K,w')) \leq \norm{\left( \norm{ (w(\sigma_i),w(\tau_i)) - (w'(\sigma_i),w'(\tau_i))} \right)_i }_q\\
        = \norm{ \left( \norm{ (w(\sigma_i) - w'(\sigma_i), w(\tau_i) - w'(\tau_i)) }_q \right)_i }_q = \norm{ \left( w(\sigma) - w'(\sigma) \right)_{\sigma \in K} }_q =  \norm{w-w'}_q.
    \end{multline*}
\end{proof}

\begin{theorem} \label{thm:stability-A}
    Let $1 \leq q \leq \infty$.
    Consider $W$ with the metric obtained from the $q$-norm on $\R^{2^m}$ and 
    $\gradedPDset$ with the $q$-Wasserstein distance.
    Then 
    $\PD(K,-): W \to \gradedPDset$ is $1$-Lipschitz.
\end{theorem}

\begin{proof}
    Let $w',w'' \in W$.
    Let $w:[a,b] \to W$ be a geodesic whose image is the line segment from $w'$ to $w''$. 
    Combining \cref{lem:pl-partition}, \cref{thm:PL-w-PD}, and \cref{prop:stability-main}, the result follows.
\end{proof}

\begin{theorem} \label{thm:stability-3}
    Consider $W$ with the metric obtained from the $\infty$-norm on $\R^{2^m}$ and 
    $\PDset$ with the 
    $1$-Wasserstein distance 
    distance.
    Then
    $\PD_k(K,-): W \to \PDset$ is $2 f_k^+$-Lipschitz.
\end{theorem}

\begin{proof}
    Let $w', w'' \in W$.
    Let $w:[a,b] \to W$ be a geodesic whose image is the line segment from $w'$ to $w''$. 
    Then $w$ is a PL curve.
    Combining \cref{lem:cardinality-BDk}, \cref{lem:pl-partition}, \cref{prop:pl-main}, and \cref{thm:PL-w-PD},
    the result follows.
\end{proof}

Let $\DM$ denote the set of symmetric $m \times m$ matrices $M$ satisfying $M_{ii} \leq M_{ij}$ for all $i,j$. 

\begin{theorem} \label{thm:stability-M-1}
    Consider $\DM$ with the metric obtained from the $\infty$-norm on $\R^{m^2}$ and 
    $\gradedPDset$ with the 
    bottleneck distance.
    Then 
    $\PD(K,-): \DM \to \gradedPDset$ is $1$-Lipschitz.
\end{theorem}

\begin{proof}
    Let $M', M'' \in \DM$. 
    Let $M:[a,b] \to \DM$ be a geodesic whose images is the line segment from $M'$ to $M''$.
    Using $w_{\bullet}: \DM \to W$ \eqref{eq:wM}, we have $w_{M_{\bullet}}: [a,b] \to W$ given by the composition $w_{\bullet} \circ M_{\bullet}$. The result follows from \eqref{eq:wM} and \cref{thm:stability-A}.
\end{proof}

\begin{theorem} \label{thm:stability-M-2}
    Consider $\DM$ with the metric obtained from the $1$-norm on $\R^{m^2}$ and 
    $\PDset$ with the 
    $1$-Wasserstein distance.
    Then 
    $\PD_k(K,-): \DM \to \PDset$ is $\binom{m-1}{k}$-Lipschitz.
\end{theorem}

\begin{proof}
    $K$ has $\binom{m}{k+1}$ $k$-simplices.
    Let $M',M''\in \DM$. 
    Let $M:[a,b] \to \DM$ be a geodesic from $M'$ to $M''$ that changes the entries one at a time. 
    Using $w_{\bullet}: \DM \to W$ \eqref{eq:wM}, we have $w_{M_{\bullet}}: [a,b] \to \PD$ given by the composition $w_{\bullet} \circ M_{\bullet}$.
    By \eqref{eq:wM}, 
    each diagonal entry in $M_t$ determines the weight of at most $\binom{m-1}{k}$ $k$-simplices in $w_{M_t}$, and
    each off-diagonal entry in $M_t$ determines the weight of at most $\binom{m-2}{k-1}$ $k$-simplices in $w_{M_t}$.
    The result now follows from \cref{thm:stability-A}.
\end{proof}

\begin{theorem} \label{thm:stability-M-3}
    Consider $\DM$ with the metric obtained from the $\infty$-norm on $\R^{m^2}$ and 
    $\PDset$ with the 
    $1$-Wasserstein distance.
    Then
    $\PD_k(K,-): \DM \to \PDset$ is $2 f_k^+$-Lipschitz.
\end{theorem}

\begin{proof}
    The result follows from \eqref{eq:wM} and \cref{thm:stability-3}.
\end{proof}

Let $\tilde{K}$ be an abstract simplicial complex with set of vertices $[m]$.
A \emph{weight} on $\tilde{K}$ is an order-preserving map $\tilde{w}:(\tilde{K},\subseteq)\to (\R,\leq)$.

\begin{lemma} \label{rem:extend-L-to-K}
    Let $\tilde{K}$ be an abstract simplicial complex with vertices $[m]$, together with an order-preserving map $\tilde{w}:(\tilde{K},\subseteq) \to (\R,\leq)$.
    Then $\tilde{w}$ may be extended to a weight $w$ on the augmented combinatorial $(m-1)$-simplex $K$ by defining $w(\emptyset) = \min(\tilde{w}) - (\max(\tilde{w}) - \min(\tilde{w}))$ and for each $\sigma \in K - \tilde{K}$, $w(\sigma) = \max(\tilde{w}) + (\max(\tilde{w}) - \min(\tilde{w}))$.
\end{lemma}

If the weighted combinatorial $(m-1)$-simplex $(K,w)$ is obtained from a weighted finite simplicial complex $(\tilde{K},\tilde{w})$ via \cref{rem:extend-L-to-K}, then we may obtain the persistence diagrams of $(\tilde{K},\tilde{w})$, denoted $\PD_k(\tilde{K},\tilde{w})$, from the persistence diagrams of $(K,w)$, by replacing values of $\max(w)$ with $\infty$, and by removing the unique element $(\min(w),b)$ of $\PD_{-1}(K,w)$ and adding the element $(b,\infty)$ to $\PD_0(\tilde{K},\tilde{w})$.
Note that these persistence diagrams of $(\tilde{K},\tilde{w})$ will be elements of the free commutative monoid on the set $(-\infty,\infty]^2_\leq = \{(x,y) \ | \ x,y \in (-\infty,\infty], x \leq y\}$. Let $\PDsetextended$ denote the set of these persistence diagrams. The Wasserstein distances extend to $\PDsetextended$, where they are extended metrics.
Again, we can combine these persistence diagrams to obtain a persistence diagram $\PD(\tilde{K},\tilde{w}) \in D(\Z \times (-\infty,\infty]^2_\leq)$.
Let $\gradedPDsetextended$ denote the set of these persistence diagrams.

Let $\tilde{W} = \{w \in \R^{\tilde{K}} \ | \ w(\sigma) \leq w(\tau), \forall \sigma \subseteq \tau \in \tilde{K} \}$, which we call the \emph{weight space} of $\tilde{K}$.

\begin{theorem}[\cite{cohen2006vines,Skraba:2020}] \label{thm:stability-p}
    Let $1 \leq q \leq \infty$.
    Consider $\tilde{W}$ with the metric obtained from the $q$-norm on $\R^{\tilde{K}}$ and 
    $\gradedPDsetextended$ with the 
    $q$-Wasserstein distance.
    Then 
    $\PD(\tilde{K},-): \tilde{W} \to \gradedPDsetextended$ is $1$-Lipschitz.
\end{theorem}

\begin{proof}
    Let $\tilde{w}',\tilde{w}'' \in \tilde{W}$. 
    Let $K$ denote the augmented combinatorial $(m-1)$-simplex and
    let $W\subset \R^{2^m}$ denote the weight space of $K$. 
    By \cref{rem:extend-L-to-K}, extend $\tilde{w}', \tilde{w}''$ to $w',w'' \in W$.
    The result follows from \cref{thm:stability-A}.
\end{proof}

\begin{theorem}
    Let $0 \leq k \leq m-1$.
    Consider $\tilde{W}$ with the metric obtained from the $\infty$-norm on $\R^{\tilde{K}}$ and 
    $\PDsetextended$ with the 
    $1$-Wasserstein distance.
    Then
    $\PD_k(\tilde{K},-): \tilde{W} \to \PDsetextended$ is $2 f_k^+$-Lipschitz.
\end{theorem}

\begin{proof}
    Let $\tilde{w}',\tilde{w}'' \in \tilde{W}$. 
    Extend $\tilde{w}', \tilde{w}''$ to $w',w'' \in W$.
    The result follows from \cref{thm:stability-3}.
\end{proof}

\begin{remark}
    These results easily generalize to cell complexes~\cite{Skraba:2020}.
\end{remark}

\section{Time series and sliding windows}\label{sec:time-series-sliding-windows}

In this section we discuss our primary motivating example, namely multivariate time series and their sliding window embeddings.

\subsection{Time series} \label{sec:time-series}

We start by introducing notation for time series and their distance matrices.

Our data consists of a finite sequence $(x_1,\ldots,x_n)$ in $\R^p$, called a \emph{time series}.
We interpret a time series as $n$ observations of $p$ variables.
We organize our data using matrix notation as follows.
For $1 \leq i \leq n$, we have 
$x_i  = \begin{pmatrix} x_{i1} & x_{i2} & \cdots & x_{ip}\end{pmatrix}^\tran \in \R^p$, and let
\begin{equation*}
    X = \begin{bmatrix}
        x_{1 1} & x_{1 2} & \cdots & x_{1 p} \\
        x_{2 1} & x_{2 2} & \cdots & x_{2 p} \\
        \vdots & \vdots & \ddots & \vdots \\
        x_{n 1} & x_{n 2} & \cdots & x_{n p} 
    \end{bmatrix}.
\end{equation*} 
We consider $(x_1,\ldots,x_n)$ to be a finite ordered metric space with the metric induced by the $1$-norm on $\R^p$.
Let $D$ be the corresponding distance matrix $(d(x_k,x_\ell))_{1 \leq k,\ell \leq n}$,
where $d$ is the metric induced by the $1$-norm.

For each $1 \leq j \leq p$, we have a \emph{component time series} 
$(x_{1j}, x_{2j}, \ldots, x_{nj})$ in $\R$. 
We let $x^j = \begin{pmatrix} x_{1j} & x_{2j} & \cdots & x_{nj} \end{pmatrix}^\tran \in \R^n$.
Note that $X = \begin{pmatrix} x^1 & x^2 & \cdots & x^p \end{pmatrix}$.
We consider each $x^j$
to be a finite ordered metric space with $d(x,y) = \abs{x-y}$.
Let $D^j$ be the corresponding distance matrix $(d(x_{k j},x_{\ell j}))_{1 \leq k,\ell \leq n}$.
For $1 \leq p \leq q \leq n$, let 
$D^j[p,\ldots,q; p,\ldots, q]$
denote the submatrix of $D^j$ that is the distance matrix for the subsequence $(x_{p j}, \ldots, x_{q j})$.

\subsection{Sliding windows} \label{sec:sliding-windows}

Next, we extend our notation to sliding windows.

Given an $n \times p$ matrix $X$ of $n$ observations in $p$ variables and $L \in \N$,
we define a new matrix $\bar{X}$ whose observations are the $L$ consecutive observations of $X$.
We call $L$ the \emph{window length} and $\bar{X}$ the \emph{sliding window matrix}.
In detail, $\bar{X}$ is the $(n-L+1) \times Lp$  matrix defined as follows.
\begin{equation*}
    X = 
    \begin{bmatrix}
        x_1^\tran \\ x_2^\tran \\ \vdots \\ x_n^\tran
    \end{bmatrix}
    \quad \quad \quad \quad 
    \overline{X} = 
    \begin{bmatrix}
        x_1^\tran & x_2^\tran & \cdots & x_L^\tran\\
        x_2^\tran & x_3^\tran & \cdots & x_{L+1}^\tran\\
        \vdots & \vdots & \ddots & \vdots\\
        x_{n-L+1}^\tran & x_{n-L+2}^\tran & \cdots & x_n^\tran
    \end{bmatrix}
\end{equation*}
The \emph{sliding window time series} is given by the rows of $\overline{X}$. 
That is, it is the sequence $(\overline{x}_1, \overline{x}_2, \ldots, \overline{x}_{n-L+1})$, where 
$\overline{x}_i = \begin{pmatrix} x_i^\tran & x_{i+1}^\tran & \cdots & x_{i+L-1}^
\tran \end{pmatrix}^\tran \in \R^{Lp}$.
As above, we consider it to be a finite ordered metric space with the metric induced by the $1$-norm on $\R^{L p}$.
Let $\overline{D}$ denote the corresponding distance matrix $(d(\overline{x}_k,\overline{x}_\ell))_{1 \leq k,\ell \leq n-L+1}$.

Similarly, for each $1 \leq j \leq p$, we have the \emph{component sliding window matrix} $\overline{X}^j$, which is the $(n-L+j) \times L$ matrix given as follows.
\begin{equation*}
    X^j = 
    \begin{bmatrix}
        x_{1 j} \\ x_{2 j} \\ \vdots \\ x_{n j}
    \end{bmatrix}
    \quad \quad \quad \quad 
    \overline{X}^j = 
    \begin{bmatrix}
        x_{1 j} & x_{2 j} & \cdots & x_{L j}\\
        x_{2 j} & x_{3 j} & \cdots & x_{{L+1}, j}\\
        \vdots & \vdots & \ddots & \vdots\\
        x_{{n-L+1}, j} & x_{{n-L+2}, j} & \cdots & x_{n j}
    \end{bmatrix}
\end{equation*}
The \emph{component sliding window time series} is given by the rows of $\overline{X}^j$.
That is, it is the sequence $(\overline{x}^j_1, \overline{x}^j_2, \ldots, \overline{x}^j_{n-L+1})$, where $\overline{x}^j_i = \begin{pmatrix} x_{i j} & x_{i+1,j} & \cdots & x_{i+L-1,j} \end{pmatrix}^\tran \in \R^L$.
As above, we consider it to be a finite ordered metric space with the metric induced by the $1$-norm on $\R^L$.
Let $\overline{D}^j$ denote the corresponding distance matrix $(d(\overline{x}_k^j,\overline{x}_\ell^j))_{1 \leq k,\ell \leq n-L+1}$
where $d$ is the metric induced by the $1$-norm.

\begin{proposition} \label{lem:swe}
    We have the following relationships between the distance matrices for the time series, component time series, sliding window time series, and component sliding window time series:
    \begin{equation*}
        D = \sum_{j=1}^p D^j, \quad \overline{D} = \sum_{j=1}^p \overline{D}^j, \quad \text{and} \quad
        \overline{D}^j = \sum_{q=0}^{L-1} 
        D^j[1+q,\ldots,n-L+q; 1+q,\ldots,n-L+q]
    \end{equation*}
\end{proposition}

\begin{proof}
    The first two identities follow from the definition of the $1$-norm.
    For the third,
    \begin{multline*}
        \overline{D}_{k \ell}^j 
        = d(\overline{x}_k^j,\overline{x}_\ell^j) 
        = \sum_{q=0}^{L-1} \abs{x_{k+q,j} - x_{\ell+q,j}} 
        = \sum_{q=0}^{L-1} D_{k+q,\ell+q}^j \\
        = \sum_{q=0}^{L-1} 
        D^j[1+q,\ldots,n-L+q; 1+q,\ldots,n-L+q]_{k \ell}.
    \end{multline*}
\end{proof}

\cref{lem:swe} says that when using the $1$-norm, the distance matrix for the sliding window embedding of a multivariate time series is the sum of the distance matrices of the sliding window embeddings of the component time series. 
Furthermore, the distance matrix for the sliding window embedding is a convolution of the distance matrices for the sliding window embedding with a window length of one.

\subsection{Linear-combination time series}

We now consider linear combinations in our constructions.

Consider the time series $(x_1,\ldots,x_n)$ in $\R^p$ and let $v \in \R^p$.
We define the \emph{linear-combination time series} $(v \odot x_1, \ldots, v \odot x_n)$ in $\R^p$, where
$v \odot x_i$ is the Hadamard product defined as $\begin{pmatrix} v_1 x_{i 1} & v_2 x_{i 2} & \cdots & v_p x_{i p} \end{pmatrix}^\tran$.
Following the notation in Sections~\ref{sec:time-series} and \ref{sec:sliding-windows}, we 
define 
\begin{equation*}
    vX = \begin{bmatrix}
        v_1 x_{1 1} & v_2 x_{1 2} & \cdots & v_p x_{1 p} \\
        v_1 x_{2 1} & v_2 x_{2 2} & \cdots & v_p x_{2 p} \\
        \vdots & \vdots & \ddots & \vdots \\
        v_1 x_{n 1} & v_2 x_{n 2} & \cdots & v_p x_{n p} 
    \end{bmatrix}.
\end{equation*} 
For $1 \leq j \leq p$, we have the \emph{linear-combination component time series} $(v_j x_{1 j}, \ldots, v_j x_{n j})$ in $\R$.
It has distance matrix $\abs{v_j} D^j$.
The linear-combination time series has distance matrix is $\sum_{j=1}^p \abs{v_j} D^j$, which we denote $vD$.
For $1 \leq j \leq p$, we have the \emph{linear-combination component sliding window time series}
$(v_j \overline{x}_1^j, \ldots, v_j \overline{x}_{n-L+1}^j)$ in $\R$,
which has distance matrix $\abs{v_j} \overline{D}^j$.
Furthermore, we have
the \emph{linear-combination sliding window time series} $(v \odot \overline{x}_1, \ldots, v \odot \overline{x}_{n-L+1})$ in $\R^p$,
where we have rearranged 
$\overline{x}_i$ to be $\begin{pmatrix}
    (\overline{x}_i^1)^\tran & (\overline{x}_i^2)^\tran & \cdots & (\overline{x}_i^p)^\tran
\end{pmatrix}^\tran$ 
so that
$v \odot \overline{x}_i = \begin{pmatrix}
    (v_1 \overline{x}_i^1)^\tran & (v_2 \overline{x}_i^2)^\tran & \cdots & (v_p \overline{x}_i^p)^\tran
\end{pmatrix}^\tran$. 
It has distance matrix $\sum_{j=1}^p \abs{v_j} \overline{D}^j$, which we denote $v\overline{D}$.
This decomposition, which follows from \cref{lem:swe} applied with scaling
factors $v_j$, is the key consequence of using the $1$-norm:\ it ensures that the distance matrix of the linear-combination sliding window time series is a linear-combination of the component distance matrices, which is the object over which we optimize in \cref{problem:tfs-time-series} below.

\subsection{Sliding window persistence diagrams}

Finally, we take persistence diagrams.

Given a time series $(x_1,\ldots,x_n)$ in $\R^p$ and window length $L$,
we have the $m \times m$ sliding window distance matrix $\overline{D}$, 
where $m = n - L + 1$.
Let $K$ be the augmented combinatorial $(m-1)$ simplex and let $-1 \leq k \leq m-1$.
Then we have the degree-$k$ Vietoris--Rips persistence diagram $\PD_k(K,w_{\overline{D}})$.

Now if we take linear combinations the component time series using $v \in \R^p$, we obtain a function 
$\Psi: \R^p \to \PDset$ given by $\Psi(v) = \PD_k(K,w_{v\overline{D}})$.
Note that $v\overline{D}$ is invariant under sign changes of the coordinates of $v$.
Thus we will restrict $\Psi$ to $\left(\R_{\geq 0 }\right)^p$.
Also note that for $\lambda \geq 0$, $(\lambda v)\overline{D} = \lambda (v\overline{D})$, and for any weight $w$ and distance matrix $d$, $w_{\lambda d} = \lambda w_d$ and $\PD_k(K,\lambda w) = \lambda \PD_k(K,w)$.
Therefore $\Psi(\lambda v) = \lambda \Psi(v)$.
Hence, we will further restrict $\Psi$ to the standard geometric $(p-1)$-simplex 
$\{v \in \R_{\geq 0}^p \ | \ \sum_{i=1}^p v_i = 1\}$, which we denote $\abs{\Delta^{p-1}}$.

\section{Topological feature selection}\label{sec:feature-selection}

In this section we introduce optimization problems which we call topological feature selection. 
For our motivating example of a multivariate time series, the problem is to find the convex combination of the variables that maximizes a functional on the persistence diagram of the corresponding convex-combination distance matrix. 
More generally, we are given a finite collection of weights and we search for the convex combination of these weights that maximizes a functional on the persistence diagram of the weighted augmented combinatorial simplex.

\begin{problem}[Topological feature selection for a time series]
\label{problem:tfs-time-series}
\textnormal{
    Let $(x_1,x_2,\ldots,x_n)$  be a time series in $\R^p$ and fix a window length $L$.
    Let $\overline{D}$ denote the distance matrix for the corresponding sliding window time series.
    Let $m = n-L+1$, 
    let $K$ be the augmented combinatorial $(m-1)$-simplex. 
    Define 
    $\Psi \colon \abs{\Delta^{p-1}} \to \gradedPDset$ by
    $\Psi(v) = \PD(K,w_{v\overline{D}})$.
    Given $F:\gradedPDset \to \R$, we have the following optimization problem.
    \begin{gather*}
        \text{maximize}_v \quad (F \circ \Psi)(v)\\
        \text{subject to} \quad v \in \abs{\Delta^{p-1}}
    \end{gather*}
}
\end{problem}

For example, $F$ may map a persistence diagram to the sum of the lifetimes of all points, called \emph{total persistence}, or it may map the diagram to the lifetime of the point furthest from the diagonal, called \emph{maximum persistence}.

We can also consider the following more general optimization problem.

\begin{problem}[Topological feature selection for $p$ 
matrices]
\label{problem:tfs-matrices}
\textnormal{
Recall that $\DM$ is the set of symmetric $m \times m$ matrices $M$ satisfying $M_{ii} \leq M_{ij}$ for all $i,j$.    
Consider $M^1,\ldots,M^p \in \DM$.
    For $v \in \abs{\Delta^{p-1}}$, we have $vM = \sum_{i=1}^p v_i M^i$.
    Let $K$ be the augmented combinatorial $(m-1)$ simplex. 
    Define $\Psi:\abs{\Delta^{p-1}} \to \gradedPDset$ by 
        $\Psi(v) = \PD(K,w_{vM})$.
    Given $F:\gradedPDset \to \R$, we have the following optimization problem.
    \begin{gather*}
        \text{maximize}_v \quad (F \circ \Psi)(v)\\
        \text{subject to} \quad v \in \abs{\Delta^{p-1}}
    \end{gather*}
}
\end{problem}

In fact, we will consider the following even more general optimization problem. 

\begin{problem}[Topological feature selection for $p$ weights]
\label{problem:tfs-weights}
\textnormal{
    Let $K$ be the augmented combinatorial $(m-1)$-simplex.
    Let $W \subset \R^{2^m}$ be the weight space for $K$.
    Assume we have weights $w = (w_1,\ldots,w_p) \in W^p$.
    Such weights induce a map $\abs{\Delta^{p-1}} \to W$ given by mapping $v=(v_1,\ldots,v_p)$ to $\sum_{i=1}^p v_i w_i$, which we denote $v \cdot w$.
    Define $\Psi: \abs{\Delta^{p-1}} \to \gradedPDset$ by 
        $\Psi(v) = \PD(K,v \cdot w)$.
    Given $F:\gradedPDset \to \R$, we have the following optimization problem.
    \begin{gather*}
        \text{maximize}_v \quad (F \circ \Psi)(v)\\
        \text{subject to} \quad v \in \abs{\Delta^{p-1}}
    \end{gather*}
}
\end{problem}

\section{Optimization}
\label{sec:optimization} 

In this section, we give convergent algorithms to the topological feature selection problems in Section~\ref{sec:feature-selection} using gradient descent and also  using a more generally applicable stochastic subgradient descent.

\subsection{Topological feature selection and piecewise linearity}
\label{sec:tfs-pl}

Let $K$ be the augmented combinatorial $(m-1)$-simplex and let $W \subset \R^{2^m}$ be the weight space of $K$.
Recall that $\abs{\Delta^{m-1}} \subset \R^m$ lies in the hyperplane $\Pi = \{v \in \R^m \ | \ v_1 + \cdots + v_m = 1\}$.
This hyperplane has normal vector $\mathbf{1}=(1, 1, \dots, 1)\in \R^{m}$.
Let $\Pi_0$ denote the $(m-1)$-dimensional subspace of $\R^m$ orthogonal to $\mathbf{1}$.

Let $\trianglelefteq$ be a linear extension of the subset order on $K$
and let $\mathcal{A}$ be a choice of persistence algorithm.
Let $\mathcal{P}_{\trianglelefteq}$ be the order partition of $W$ induced by $\trianglelefteq$.
Recall that we call elements of a partition regions.
By \cref{lem:regions}, the regions are $2^m$-dimensional polyhedrons.
The total order $\leq_{(w,\trianglelefteq)}$ on $K$ is constant on regions.
We have a birth-death matching $\BD^\mathcal{A}(K,\leq_{(w,\trianglelefteq)})$ that is constant on regions
and we have the function $\PD(K,-): W \to \gradedPDset$.

\begin{definition} 
Let $\hat{\mathcal{P}}_\trianglelefteq$ be a partition of $W$ consisting of $2^m$-dimensional polyhedra that is a refinement of $\mathcal{P}_{\trianglelefteq}$.
We say that a function $F: \gradedPDset \to \R$ is \emph{PL for $\hat{P}_\trianglelefteq$} if 
$F(\PD(K,-)): W \to \R$ is continuous and is affine on each region of $\hat{\mathcal{P}}_\trianglelefteq$.
That is, $F(\PD(K,-))$ is continuous and for $A \in \hat{\mathcal{P}}_\trianglelefteq$ and for all $w \in A$, $F(\PD(K,w)) = w \cdot a + b$ for some $a \in \R^{2^m}$ and $b \in \R$.
Similarly, given $-1 \leq k \leq m-1$, we say that a function $F: \PDset \to \R$ is \emph{PL for $\hat{P}_\trianglelefteq$} if 
$F(\PD_k(K,-)): W \to \R$ is continuous and is affine on each region of $\hat{\mathcal{P}}_\trianglelefteq$.
\end{definition}

For example, the total persistence of a persistence diagram is PL for $\hat{\mathcal{P}}_\trianglelefteq= \mathcal{P}_\trianglelefteq$.
The maximum persistence of a persistence diagram is PL for a proper subpartition in which each region has a constant birth-death pair with maximum persistence.

Now, we address \cref{problem:tfs-weights} for a function $F\colon \gradedPDset \to \R$ that is PL for such a $\hat{P}_\trianglelefteq$. 
Assume  we have weights $w = (w_1,\ldots,w_p) \in W^p$.
We have the following composite map.
\begin{equation} \label{eq:tfs-map}
    F(\PD(K,- \cdot w)): \abs{\Delta^{p-1}} \xrightarrow{- \cdot w} W \xrightarrow{\PD(K,-)} \gradedPDset \xrightarrow{F} \R
\end{equation}

\begin{lemma} \label{lem:partition-pullback}
    Let $\mathcal{P}$ be a finite partition of $W$ consisting of 
    polyhedra.
    The map $- \cdot w: \abs{\Delta^{p-1}} \to W$ induces a finite partition $\mathcal{S}$ of $\abs{\Delta^{p-1}}$ consisting of polytopes (bounded polyhedra).
\end{lemma}

\begin{proof}
    Let $\mathcal{S}$ be the set consisting of the nonempty inverse images of the regions in $\mathcal{P}$ for the map $- \cdot w: \abs{\Delta^{p-1}} \to W$.
    Then $\mathcal{S}$ is partition of $\abs{\Delta^{p-1}}$ called the pullback of $\mathcal{P}$ along $- \cdot w$.
    Since $- \cdot w$ is an affine map, the regions of $\mathcal{S}$ are polyhedrons.
    Since $\abs{\Delta^{p-1}}$ is compact, the regions of $\mathcal{S}$ are polytopes.
\end{proof}

We have the following.

\begin{proposition}
    Consider \cref{problem:tfs-weights}. 
    Let $\hat{\mathcal{P}}_\trianglelefteq$ be a partition of $W$ consisting of $2^m$-dimensional polyhedra that is a refinement of $\mathcal{P}_{\trianglelefteq}$.
    Let $\mathcal{S}$ be the partition of $\abs{\Delta^{p-1}}$ given by the inverse images of the regions in $\hat{\mathcal{P}}_\trianglelefteq$ for the map $- \cdot w: \abs{\Delta^{p-1}} \to W$.
    Assume that $F$ is PL for $\hat{\mathcal{P}}_\trianglelefteq$.
    Then $F(- \cdot w): \abs{\Delta^{p-1}} \to \R$ is continuous and affine on each region of $\mathcal{S}$.
    However, in general it is not convex on $\abs{\Delta^{p-1}}$.
\end{proposition}

Let $f$ denote the negative of the map \eqref{eq:tfs-map} restricted to a region $S$ of $\mathcal{S}$.
For $v$ in $S$, and by the assumption that $F$ is PL for $\hat{P}_\trianglelefteq$, we have 
$f(v) =  \left(\textstyle\sum_{i=1}^p v_i w_i\right)\cdot a + b 
     = \textstyle\sum_{i=1}^p v_i \left(w_i\cdot a\right) +b 
     = v\cdot (w_1\cdot a, w_2\cdot a, \dots, w_p\cdot a)^\tran  + b$ 
for some $a\in \R^{2^m}$ and $b\in\R$.
Hence, the composite map \eqref{eq:tfs-map} is continuous and is affine on each region of $\mathcal{S}$.
Furthermore the gradient $f$ as a function on $\R^p$, $\nabla f = (w_1\cdot a, w_2\cdot a, \dots, w_p\cdot a)^\tran$ is constant on $S$.
However, we want the gradient of $f$ as a function on the hyperplane $\Pi$ containing $S$,
which is given by the projection of $\nabla f$ onto $\Pi_0$,
since the gradient of $f|_\Pi$ is given by the maximum directional derivative for directions in $\Pi_0$, for which the component of $\nabla f$ in direction of $\mathbf{1}$ is irrelevant. Thus, $\nabla (f|_{\Pi}) = \proj_{\Pi_0}(\nabla f) = \nabla f - \frac{1}{p}\left(\sum_i \nabla f_i\right)\mathbf{1}$. 
Thus, we update $v\mapsto v - \gamma \proj_{\Pi_0}(\nabla f)$, where $\gamma$ is some step size.

Our first algorithm, \cref{alg:gradient-path-exact}, follows the gradient path exactly. 
The cardinality of the set $\mathcal{R}$, defined in the algorithm, is at most one.
By \cref{lem:partition-pullback}, $\mathcal{S}$ consists of a finite number of polytopes. It follows that $\mathcal{R}$ is empty only if $v_j$ is on the boundary of $\abs{\Delta^{p-1}}$ and $\proj_{\Pi_0}(\nabla f)(v_j)$ points outside of $\abs{\Delta^{p-1}}$.
For $S \in \mathcal{S}$, let $\overline{S}$ denote its closure.

\begin{algorithm}
\caption{Gradient descent}\label{alg:gradient-path-exact}
\begin{algorithmic}
    \Require 
    $N \in \mathbb{N}$, 
    $w \in W^p$
    \Ensure $v_0,v_1,\ldots,v_N \in \abs{\Delta^{p-1}}$
    \State $f \gets - F(\PD(K,- \cdot w)): \abs{\Delta^{p-1}} \to \R$
    \State $v_0 \gets \left( \frac{1}{p},\cdots,\frac{1}{p} \right)$
    \For{$i = 0, \dots, N-1$}
    \State $\mathcal{R} \gets \{S \in \mathcal{S} \ | \ \exists \eps > 0$ such that $\forall t \in (0,\eps), v_i + t \proj_{\Pi_0}(\nabla f)(v_i) \in S\}$
    \If{$\abs{\mathcal{R}} = 1$}
        \State $S \gets$ unique element of $\mathcal{R}$
        \State $\hat{t} \gets \min \{ t \geq 0 \ | \ v_i + t \proj_{\Pi_0}(\nabla f)(v_i) \in \overline{S}\}$
        \State $v_{i+1} \gets v_i + \hat{t} \proj_{\Pi_0}(\nabla f)(v_i)$
    \Else
        \State $v_{i+1} \gets v_i$
    \EndIf
    \EndFor
\end{algorithmic}
\end{algorithm}

While this algorithm is mathematically appealing, following this path to (a point close to) a local maximum is too expensive in practice due to the large number of required persistent homology computations because of the large number of regions in $\mathcal{S}$.

Instead of this deterministic algorithm, we will use a stochastic algorithm. 
Furthermore, we will generalize from using gradients to subgradients, which permits the use of more general functions $F$.
We will start by introducing the general theory.

\subsection{Projected stochastic subgradient descent}
\label{sec:subgradient}

Assume that $A \subseteq \R^q$ is a closed convex set with nonempty interior,
$f: A \to \R$,
and that we want to solve the optimization problem $\minimize_{x \in A} f(x)$.

Assume that $f$ is locally Lipschitz (and hence continuous).
Then by Rademacher's theorem, $f$ is differentiable almost everywhere and thus has a well defined Clarke subdifferential, $\partial f(z)$,
given by the convex hull of the set of limits of $\nabla f(z_i)$ taken over sequences $(z_i)$ converging to $z$ for which $f$ is differentiable at $z_i$.
When $f$ is $C^1$-smooth then $\partial f(z)$ consists of only the gradient of $f$ at $z$
and when $f$ is convex then $\partial f(z)$ coincides with the subdifferential defined in convex analysis.
Say that $z$ is \emph{critical} if $0 \in \partial f(z)$.

We will use the following \emph{projected stochastic subgradient method}~\cite{MR4056927}.
Choose a positive sequence $(a_k)$, called the learning rate, 
a sequence of random variables $(\zeta_k)$ in $\R^q$, called noise, and an initialization $z_0 \in A$.
Then iterate
\begin{equation} \label{eq:stochastic-subgradient}
    z_{k+1} = \proj_A(z_k - a_k(y_k + \zeta_k)), \text{ where } y_k \in \partial f (z_k),
\end{equation}
and $\proj_A: \R^{p+1} \to A$ indicates the unique projection to the nearest point of $A$.
We assume the following mild hypotheses:
1.\ $(a_k)$ is square summable but not summable;
2.\ almost surely $\sup \norm{z_k} < \infty$; and
3.\ each $\zeta_k$, conditioned on the past, has mean $0$ and the second moment grows at a controlled rate~\cite[Assumption C]{MR4056927}.

Say that $f$ is \emph{Whitney stratifiable} if its graph has a \emph{Whitney stratification}.
The class of such function is large, containing all continuous semi-algebraic functions
and all continuous definable functions for any o-minimal structure~\cite{MR4056927}.

\begin{theorem}[\cite{MR4056927}] \label{thm:subgradient}

    Assume that $f$ is locally Lipschitz and Whitney definable.
    Apply projected stochastic subgradient descent~\eqref{eq:stochastic-subgradient} to obtain a sequence $(z_k)$.
    Then almost surely the limit points of $(z_k)$ are critical points of $f$
and the sequence $(f(z_k))$ converges.

\end{theorem}

\begin{remark}
    The projected stochastic subgradient method is a proximal extension of the stochastic subgradient method: $\minimize_{z \in A} f(z) = \minimize_{z \in \R^q} f(z) + g(z)$, where $g(z) = 0$ if $z \in A$ and $g(z)=\infty$ if $z \not\in A$ and the proximal operator of $g$ is the projection operator.
\end{remark}

\subsection{Optimizing functions involving persistence diagrams}
\label{sec:optimizing-functions}

Let $A \subset \R^q$ be a closed convex set with nonempty interior.
Recall that $\PDset$ is the set of persistence diagrams.
Let $\Psi: A \to \PDset$ and $F: \PDset \to \R$.
Consider the optimization problem $\maximize_{z \in A} F(\Psi(z))$.

By \cref{thm:subgradient}, 
if $F \circ \Psi$ is locally Lipschitz and is Whitney definable then
we may find a local minimum to this problem using projected stochastic subgradient descent.
We will show that for the problems that we will consider, $F \circ \Psi$ is piecewise linear and continuous and hence locally Lipschitz and Whitney definable. 
However, the latter are vastly more general,
containing continuous semi-algebraic functions
and continuous definable functions for any o-minimal structure.
In particular, they include Vietoris-Rips, \v{C}ech, and Delaunay filtrations, 
and persistence landscapes and persistence images~\cite{carriere2021optimizing}.

\subsection{Topological feature selection using projected stochastic gradient descent} \label{sec:tfs-sgd}

For topological feature selection, we have $\Psi: \abs{\Delta^{p-1}} \to \PDset$ given by $\PD(K,- \cdot w)$.
If we have a function $F: \gradedPDset \to \R$ such that $f = -F(\PD(K,- \cdot w)): \abs{\Delta^{p-1}} \to \R$ is locally Lipschitz and Whitney definable
then we may use projected stochastic subgradient descent to find a local minimum for $f$.
We give pseudocode for this method in \cref{alg:stochastic-path},
where $N_{p-1}(0,\sigma^2 I)$ denotes the $(p-1)$-dimensional multivariate normal distribution centered at the origin with covariance matrix given by the scalar multiple of the identity matrix by $\sigma^2$.

\begin{algorithm}
\caption{Stochastic subgradient descent}\label{alg:stochastic-path}
\begin{algorithmic}
    \Require 
    $N \in \mathbb{N}$, $\eta_0, d, \sigma > 0$, $w \in W^p$
    \Ensure $v_0,v_1,\ldots,v_N \in \abs{\Delta^{p-1}}$
    \State $f \gets -F(\PD(K,- \cdot w)): \abs{\Delta^{p-1}} \to \R$
    \State $v_0 \gets \left( \frac{1}{p},\cdots,\frac{1}{p} \right)$
    \For{$i = 0, \dots, N-1$}
    \State $\eta_{i+1} \gets \frac{\eta_0}{1 + d i}$
    \State $\zeta_i \gets N_{p-1}(0,\sigma^2 I)$
    \State $y_i \in \partial f(v_i)$
    \State $v_{i+1} \gets \proj_{\abs{\Delta^{p-1}}}(v_i - \eta_{i+1} (y_i + \zeta_i))$
    \EndFor
\end{algorithmic}
\end{algorithm}

In the cases that we will consider, there is finite partition $\mathcal{S}$ of $\abs{\Delta^{p-1}}$ into polytopes such that $f$ is affine on each polytope (see Section~\ref{sec:tfs-pl}). 
We will use the gradient of this affine function, as a function of $\R^p$,
add Gaussian noise in $\R^p$, and then project onto $\Pi_0$
(see Section~\ref{sec:tfs-pl}). 
We will use the efficient algorithm of Duchi et al~\cite{duchi2008efficient} to find the closest point in $\abs{\Delta^{p-1}}$.

\subsection{Gradient path} \label{sec:gradient-path}

For either of our algorithms, 
gradient descent results in a 
sequence $v_0,v_1,\ldots,v_N \in \abs{\Delta^{p-1}}$ starting at the barycenter
and ending at $v_N$, which we denote $s$ and call the \emph{score}.
Then $s = (s_1,\ldots,s_p)$ and for $1 \leq j \leq p$ we call $s_j$ the \emph{score} of the $j$th weight $w_j$.
Linearly interpolating between the points $v_0,v_1,\ldots,v_N$, we obtain the corresponding PL curve $v:[0,N] \to \abs{\Delta^{p-1}}$.

Furthermore, we have
$v(-) \cdot w: [0,N] \to W$, where $v(t) \cdot w = \sum_{i=1}^p v_i(t) w_i$. 
For $1 \leq q \leq \infty$, let $W$ have the metric induced by the $q$-norm on $\R^{2^m}$.
Let $C = \max_{1 \leq i,j \leq p} \norm{w_i-w_j}_q$.
By Bauer's maximum principle, $v(-) \cdot w$ is $C$-Lipschitz.
If desired, we may rescale the PL curve $v$ so that for some $T > 0$, such as $T = C N$, $v:[0,T] \to \abs{\Delta^{p-1}}$ and $v(-) \cdot w$ is $1$-Lipschitz.

\subsection{Stability}

We have the following stability results.

\begin{theorem} \label{thm:stability-FV-A}
    Let $1 \leq q \leq \infty$.
    Let $W$ have the metric induced by the $q$-norm on $\R^{2^m}$ and 
    let $\gradedPDset$ have the $q$-Wasserstein distance.
    Let $w = (w_1,\ldots,w_p) \in W^p$.
    Let $v:[0,T] \to \abs{\Delta^{p-1}}$ be a PL curve such that the PL curve $v(-) \cdot w: [0,T] \to W$ is $C$-Lipschitz.
    Let $\hat{P}_\trianglelefteq$ be a partition of $W$ consisting of $2^m$-dimensional polyhedra which is a refinement of $P_\trianglelefteq$.
    Let $F:\gradedPDset \to \R$ be PL for $\hat{P}_\trianglelefteq$.
    Assume that $F$ is $L$-Lipschitz.
    Then $F(\PD(K,v(-)\cdot w)): [0,T] \to \R$ is $C L$-Lipschitz.
\end{theorem}

\begin{proof} 
    $F(\PD(K,v(-)\cdot w))$ is the composite of the functions
    $v(-) \cdot w: [0,T] \to W$, 
    $\PD(K,-): W \to \gradedPDset$ and
    $F: \gradedPDset \to \R$.
    By assumption, the first and third of these functions are $C$-Lipschitz and $L$-Lipschitz, respectively.
    The second function is $1$-Lipschitz by \cref{thm:stability-A}.
    The result follows.
\end{proof}

\begin{corollary} \label{cor:stability-FV-A}
    Let $1 \leq q \leq \infty$.
    Let $W$ have the metric induced by the $q$-norm on $\R^{2^m}$ and 
    let $\gradedPDset$ have the $q$-Wasserstein distance.
    Let $w = (w_1,\ldots,w_p) \in W^p$.
    Let $C = \max_{1 \leq i,j \leq p} \norm{w_i-w_j}_q$.
    Let $\hat{P}_\trianglelefteq$ be a partition of $W$ consisting of $2^m$-dimensional polyhedra which is a refinement of $P_\trianglelefteq$.
    Let $F:\gradedPDset \to \R$ be PL for $\hat{P}_\trianglelefteq$.
    Assume that $F$ is $L$-Lipschitz.
    Let $v:[0,N] \to \abs{\Delta^{p-1}}$ be the PL curve produced by either \cref{alg:gradient-path-exact} or \cref{alg:stochastic-path}.
    Then $F(\PD(K,v(-)\cdot w)): [0,N] \to \R$ is $C L$-Lipschitz.
\end{corollary}

\begin{theorem} \label{thm:stability-FV-3}
    Let $-1 \leq k \leq m-1$.
    Let $W$ have the metric obtained from the $\infty$-norm on $\R^{2^m}$ and let $\PDset$ have the $1$-Wasserstein distance. 
    Let $w = (w_1,\ldots,w_p) \in W^p$.
    Let $v:[0,T] \to \abs{\Delta^{p-1}}$ be a PL curve such that the PL curve $v(-) \cdot w: [0,T] \to W$ is $C$-Lipschitz.
    Let $\hat{P}_\trianglelefteq$ be a partition of $W$ consisting of $2^m$-dimensional polyhedra which is a refinement of $P_\trianglelefteq$.
    Let $F:\PDset \to \R$ be PL for $\hat{P}_\trianglelefteq$.
    Assume that $F$ is $L$-Lipschitz.
    Then $F(\PD_k(K,v(-)\cdot w)): [a,b] \to \R$ is $2 f_k^+ C L$-Lipschitz.
\end{theorem}

\begin{proof}
    The proof is the same as the proof of \cref{thm:stability-FV-A}, expect in this case $\PD_k(K,-):W \to \PDset$ is $2 f_k^+$-Lipschitz by \cref{thm:stability-3}. 
\end{proof}

\begin{corollary} \label{cor:stability-FV-3}
    Let $W$ have the distance induced by the $\infty$-norm on $\R^{2^m}$ and
    let $
    \PDset$ have the $1$-Wasserstein distance.
    Let $w = (w_1,\ldots,w_p) \in W^p$.
    Let $C = \max_{1 \leq i,j \leq p} \norm{w_i-w_j}_{\infty}$.
    Let $\hat{P}_\trianglelefteq$ be a partition of $W$ consisting of $2^m$-dimensional polyhedra which is a refinement of $P_\trianglelefteq$.
    Let $F:\PDset \to \R$ be PL for $\hat{P}_\trianglelefteq$.
    Assume that $F$ is $L$-Lipschitz.
    Let $v:[0,N] \to \abs{\Delta^{p-1}}$ be the PL curve produced by either \cref{alg:gradient-path-exact} or \cref{alg:stochastic-path}.
    Then $F(\PD(K,v(-)\cdot w)): [0,N] \to \R$ is $2 f_k^+ C L$-Lipschitz.
\end{corollary}

\begin{example}
    Let $F: \gradedPDset \to \R$ be given by the total persistence of the persistence diagram.
    Let $\gradedPDset$ have the $1$-Wasserstein distance.
    Then for $\alpha \in \gradedPDset$, $F(\alpha) = W_1(\alpha,0)$.
    By the reverse triangle inequality, for $\alpha, \beta \in \gradedPDset$, $\abs{F(\alpha) - F(\beta)} \leq W_1(\alpha,\beta)$.
    That is, $F$ is $1$-Lipschitz.
    Let $w = (w_1,\ldots,w_p) \in W^p$.
    Let $C = \max_{1 \leq i,j \leq p} \norm{w_i-w_j}_{1}$.
    Let $v:[0,N] \to \abs{\Delta^{p-1}}$ be the PL curve produced by either \cref{alg:gradient-path-exact} or \cref{alg:stochastic-path}.
    Let $V:[0,N] \to \gradedPDset$ be given by $\PD(K,v(-) \cdot w)$.
    By \cref{cor:stability-FV-A}, $F \circ V : [0,N] \to \R$ is $C$-Lipschitz.
\end{example}

More generally, we have the following.

\begin{example} \label{ex:gradient-path-lipschitz-general}
    Let $\ell \in \mathbb{N}$.
    Let $F: \gradedPDset \to \R$ be given by the sum of the persistences of the $\ell$ most-persistent points in the persistence diagram.
    Let $\gradedPDset$ have the $1$-Wasserstein distance.
    By the definition of $1$-Wasserstein distance~\cite{csehm:lipschitz},
    for $\alpha, \beta \in \gradedPDset$, $\abs{F(\alpha) - F(\beta)} \leq W_1(\alpha,\beta)$.
    That is, $F$ is $1$-Lipschitz.
    Let $w = (w_1,\ldots,w_p) \in W^p$.
    Let $C = \max_{1 \leq i,j \leq p} \norm{w_i-w_j}_{1}$.
    Let $v:[0,N] \to \abs{\Delta^{p-1}}$ be the PL curve produced by either \cref{alg:gradient-path-exact} or \cref{alg:stochastic-path}.
    Let $V:[0,N] \to \gradedPDset$ be given by $\PD(K,v(-) \cdot w)$.
    By \cref{cor:stability-FV-A}, $F \circ V : [0,N] \to \R$ is $C$-Lipschitz.
\end{example}

\subsection{Mean gradient path} \label{sec:mean-gradient-path}

For our topological feature selection problems, we have two gradient paths (Section~\ref{sec:gradient-path}), one from gradient descent and the other from stochastic gradient descent.
For each of these, we will produce an average path the satisfies a law of large numbers and a central limit theorem.

For the stochastic gradient descent, 
choose $N$, fix the learning rate ($\eta_0, d > 0$) and size of the Gaussian noise ($\sigma > 0$).
Then apply \cref{alg:stochastic-path} as in Section~\ref{sec:tfs-sgd}.
We obtain a sequence $v = (v_0,v_1,\ldots,v_N) \subset \abs{\Delta^{p-1}}$ (see Section~\ref{sec:gradient-path}).
Consider $\zeta = (\zeta_1,\ldots,\zeta_N)$ as a random variable,
$\zeta: (\Omega,P) \to \R^{N(p-1)}$, where $(\Omega,P)$ is an abstract probability space, and $v$ as a function of $\zeta$.
Then we have the \emph{random gradient path},
\begin{equation}
    v(\zeta): (\Omega,P) \xrightarrow{\zeta} \R^{N(p-1)} \xrightarrow{v} \R^{(N+1)p}.
\end{equation}
Note that $(v(\zeta))_i = v_i(\zeta)$.

\begin{remark}
    Instead of using fixed Gaussian noise, one may use noise satisfying the much weaker assumption of Section~\ref{sec:subgradient}.
\end{remark}

For gradient descent, choose $N$.
Then apply \cref{alg:gradient-path-exact} to obtain a function $v: W^p \to \abs{\Delta^{p-1}}^{N+1} \subset \R^{p(N+1)}$, where $v = (v_0,\ldots,v_N)$, $v_i: W^p \to \abs{\Delta^{p-1}} \subset \R^p$.

\begin{lemma} \label{lem:continuous-piecewise}
    Let $0 \leq i \leq N$. 
    The map $v_i: W^p \to \abs{\Delta^{p-1}}$ is piecewise continuous.
\end{lemma}

Let $\zeta: (\Omega,P) \to W^p$ be a random variable with values in $W^p \subset \R^{p2^m}$, where $(\Omega,P)$ is an abstract probability space.
By \cref{lem:continuous-piecewise}, for $0 \leq i \leq N$, 
$v_i(\zeta) = v_i \circ \zeta$ is a random variable with values in $\abs{\Delta^{p-1}} \subset \R^p$
and
$v(\zeta) = v \circ \zeta$ is a random variable with values in $\abs{\Delta^{p-1}}^{N+1} \subset \R^{p(N+1)}$.

For both stochastic gradient descent and gradient descent,
since $\abs{\Delta^{p-1}}$ is bounded, all moments of $v_i(\zeta)$ and $v(\zeta)$ are finite.
For a random variable $Y$, let $E(Y)$ denote its expectation.
We refer to the sequence
$E(v(\zeta)) = (E(v_0(\zeta)),\ldots,E(v_N(\zeta)))$ as the \emph{mean gradient path}.
We also use this term to refer to the corresponding PL path $[0,N] \to \abs{\Delta^{p-1}}$.
As a special case, $E(v_N(\zeta))$ is the \emph{mean score} of $w = (w_1,\ldots,w_p)$ and 
for $1 \leq j \leq p$, we call $s_j$ the \emph{mean score} of the $j$th weight $w_j$.

Since we have finite first moments, we have a strong law of numbers, 
which says that empirical mean gradient paths converge to the mean gradient path with probability one.
As a special case, the empirical mean score converges to the mean score with probability one.

\begin{theorem}[Strong Law of Large Numbers] \label{thm:slln}
    Suppose the we have a sequence of independent $N(p-1)$-dimensional normally distributed random variables, $\zeta^{(1)},\zeta^{(2)},\ldots \sim N(0,\sigma^2 I)$.
    Then 
    for $1 \in 1,\ldots, N$,
    \begin{gather*}
        \frac{1}{n}\left( v_i(\zeta^{(1)}) + \cdots + v_i(\zeta^{(n)}) \right) \to E( v_i(\zeta^{(1)}) ) \text{ almost surely, and}\\
        \frac{1}{n}\left( v(\zeta^{(1)}) + \cdots + v(\zeta^{(n)}) \right) \to E( v(\zeta^{(1)}) ) \text{ almost surely.}
    \end{gather*}
    That is, 
    $P(\frac{1}{n}\left( v_i(\zeta^{(1)}) + \cdots + v_i(\zeta^{(n)}) \right) \to E( v_i(\zeta^{(1)}) )) = 1$
    and
    $P(\frac{1}{n}\left( v(\zeta^{(1)}) + \cdots + v(\zeta^{(n)}) \right) \to E( v(\zeta^{(1)}) )) = 1$.
\end{theorem}

Since we have finite second moments, we also have a central limit theorem.

\begin{theorem}[Central Limit Theorem] \label{thm:clt}
    Suppose the we have a sequence of independent $N(p-1)$-dimensional normally distributed random variables, $\zeta^{(1)},\zeta^{(2)},\ldots \sim N(0,\sigma^2 I)$.
    Then 
    for $1 \in 1,\ldots, N$,
    \begin{equation*}
        \sqrt{n} \left( \frac{1}{n} \left( v_i(\zeta^{(1)}) + \cdots + v_i(\zeta^{(n)}) \right) - E(v_i(\zeta^{(1)})) \right) \to N(0,\Sigma_i) \text{ in distribution,}
    \end{equation*}
    where $\Sigma_i$ denotes the covariance matrix of $v_i(\zeta^{(1)})$
    and 
    \begin{equation*}
        \sqrt{n} \left( \frac{1}{n} \left( v(\zeta^{(1)}) + \cdots + v(\zeta^{(n)}) \right) - E(v(\zeta^{(1)})) \right) \to N(0,\Sigma) \text{ in distribution,}
    \end{equation*}
    where $\Sigma$ denotes the covariance matrix of $v(\zeta^{(1)})$.
\end{theorem}

\section{Applications} \label{sec:examples}

In this section, we consider synthetic and biological data sets and apply our method to assign scores to individual time series based on their contribution to global dynamics. 
For each example, if an update step leaves the simplex, we apply an orthogonal projection back to the simplex. 

In each example, we choose the initial learning rate $\eta_0$ to be such that the first step of the optimization has length $10^{-3}$. We choose the decay rate $d$ such that the final step length is approximately $\tfrac{1}{2} 10^{-3}$. The gradient-noise scale $\sigma$ is set to one tenth of the norm of the projected subgradient barycenter so that the injected noise is an order of magnitude smaller than the initial gradient signal.

\subsection{Synthetic data}

We first consider two synthetic data sets.

\subsubsection{Three sine curves with varying amounts of noise}

This is an instance of \cref{problem:tfs-time-series} with $p = 3$ and $n = 300$.
Consider the time series $f\colon \R\to\R^3$ with components 
\[f_i(t)=\sin\left(\tfrac{2\pi}{50}(t-K_i)\right)\hspace{1cm} i=1,2,3,\]
where each $K_i\in\R$ is a randomly chosen phase shift.
We restrict to the finite domain $D=\{ 1, \dots, 300\}$ and add iid Gaussian noise with standard deviation $\tau=1$, $\tau = 2$, and $\tau = 3$ to the signals, respectively. These time series and the underlying sines are displayed in Figure~\ref{fig:noisy_sines}~(left). 
With window length $L=250$, we obtain the sliding window point cloud consisting of $51$ points in $\R^{750}$. 
This window length (a large multiple of the period) is chosen to maximize maximum persistence. 
The augmented combinatorial
simplex $K$ therefore has dimension $50$, and the optimization variable
$v \in \abs{\Delta^{p-1}} = \abs{\Delta^2}$ represents a convex combination of the three pairwise distance matrices.
We take $F$ to be the maximum persistence of the degree-$1$ persistence diagram.

\begin{figure}[ht]
\centering
\includegraphics[width=1\textwidth]{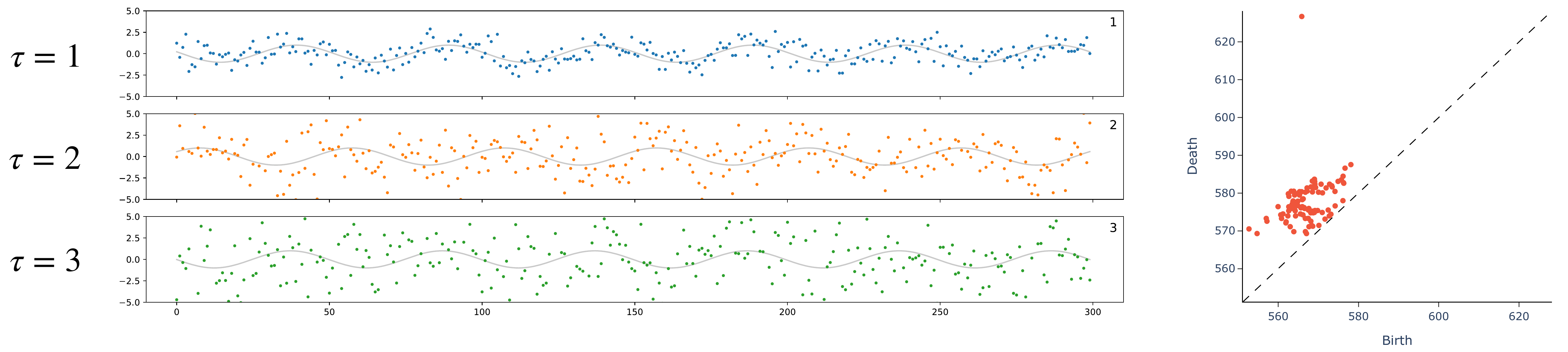}
\caption{Samples from three sine waves together with varying amounts of iid Gaussian noise (left) and the persistence diagram in degree one of the high-dimensional embedded data (right).}\label{fig:noisy_sines}
\end{figure}

Next, we compute the degree-$1$ persistent homology (Figure~\ref{fig:noisy_sines}, right). 
Each embedded sine wave traces a loop in $\R^{750}$~\cite{perea2015sliding}.
This loop manifests as a single prominent point in the degree-$1$ persistence diagram; its lifetime  is larger for cleaner signals.
Since all three components share the same period, the convex-combination distance matrix at any $v \in \abs{\Delta^2}$ also yields a single loop rather than multiple independent ones, and the degree-$1$ diagram (Figure~\ref{fig:noisy_sines}, right) correspondingly shows one outstanding feature.

Initializing at the barycenter $\lambda=(\tfrac13,\tfrac13,\tfrac13)$, we optimize for maximum persistence using the projected stochastic subgradient descent of Algorithm~\ref{alg:stochastic-path} with initial learning rate $\eta_0\approx 6\times 10^{-5}$, decay factor $d=10^{-3}$, and gradient-noise scale $\sigma\approx 1.6$. The noise scale $\sigma$ is set to one tenth of the norm of the projected subgradient at the barycenter. With these choices the step length is $10^{-3}$ at the first iterate and decays to roughly half of this value by the final iterate. A path through the geometric $2$-simplex leading to a local maximum is visualized in Figure~\ref{fig:noisy_sines_path} (left). The algorithm reaches the vertex corresponding to the least-noisy signal after $530$ steps. The numerical value of the objective function throughout the descent is shown in Figure~\ref{fig:noisy_sines_path} (right).

\begin{figure}[ht]
\centering
\includegraphics[width=0.75\textwidth]{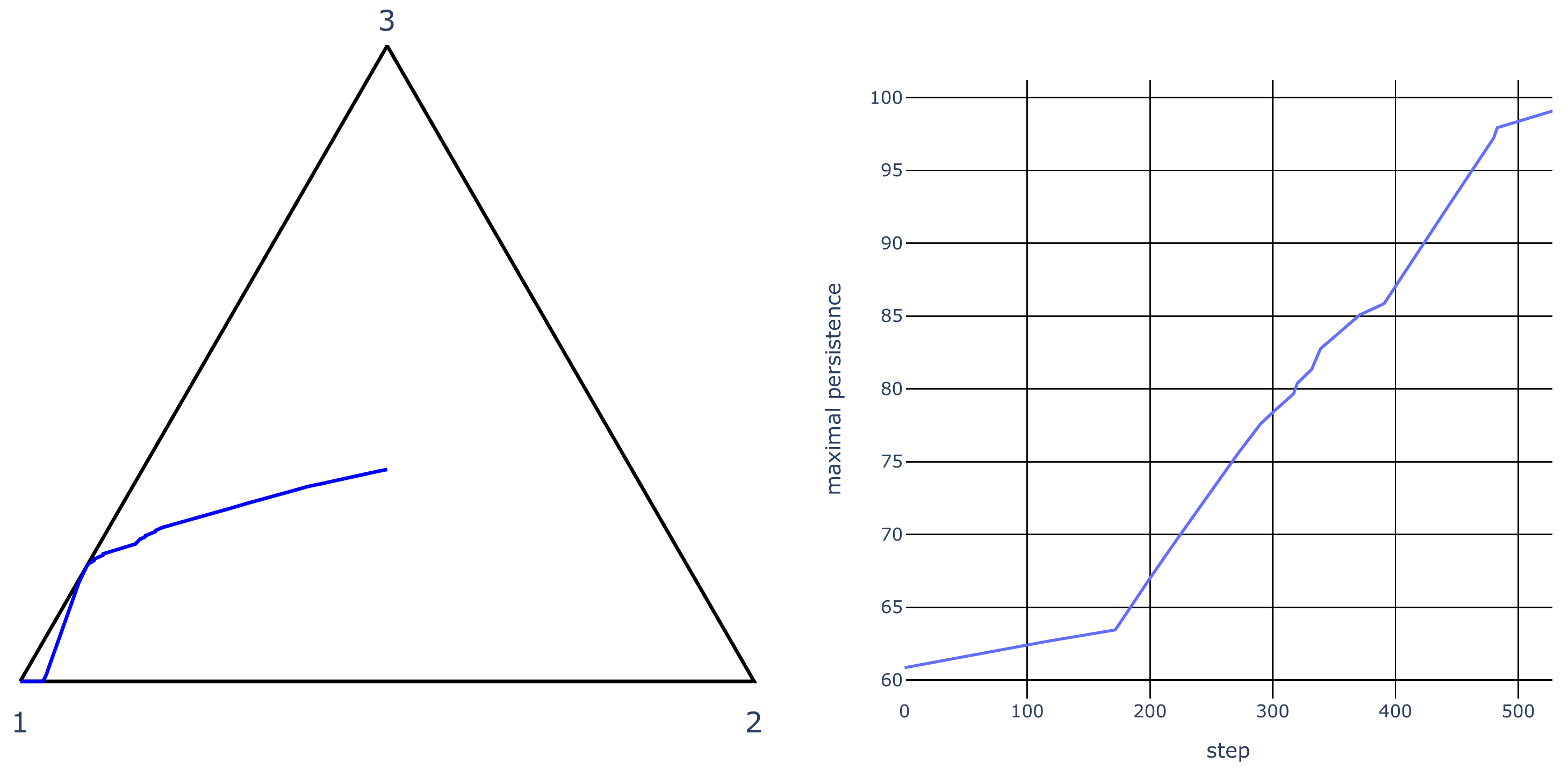}
\caption{A visualization of the gradient path (left) and the graph of the functional that is being optimized (right).
The algorithm begins at the barycenter $(\tfrac{1}{3},\tfrac{1}{3},\tfrac{1}{3})$ of the $2$-simplex and converges to the vertex corresponding to the least-noisy signal. 
}\label{fig:noisy_sines_path}
\end{figure}

\subsubsection{Detecting periodicity}\label{ssec:periodicity}
This is an instance of \cref{problem:tfs-time-series} with $p = 10$ and $n = 300$.
Consider the time series $f\colon \R\to\R^{10}$ with components 
\[f_i(t)=\sin\left(\tfrac{2\pi}{50}(t-K_i)\right)\hspace{1cm} i=1,2,\dots, 10,\]
where each $K_i\in\R$ is a randomly chosen phase shift.
We restrict to the finite domain $D=\{ 1, \dots, 300\}$ (Figure~\ref{fig:synth_data}, (a)). 
We randomly permute the values of seven of the signals (Figure~\ref{fig:synth_data}, (b)). 
We then add iid Gaussian noise with standard deviation $\tau=1.5$ to each signal (Figure~\ref{fig:synth_data}, (c)). 

With window length $L=250$, we obtain the sliding window point cloud consisting of $51$ points in $\R^{2500}$. 
As in the previous example, the simplex $K$ has dimension $50$. The optimization variable $v \in \abs{\Delta^9}$ represents a convex combination of ten pairwise distance matrices.
We take $F$ to be the maximum persistence of the degree-$1$ persistence diagram.
Without \emph{a priori} knowledge of the underlying period of the signals, an appropriate window length may be identified by searching for the value that approximately maximizes the maximum persistence in the resulting persistence diagrams. In this case, the value $L=250$ maximizes the maximum persistence.

\begin{figure}[ht]
\centering
\includegraphics[width=\textwidth]{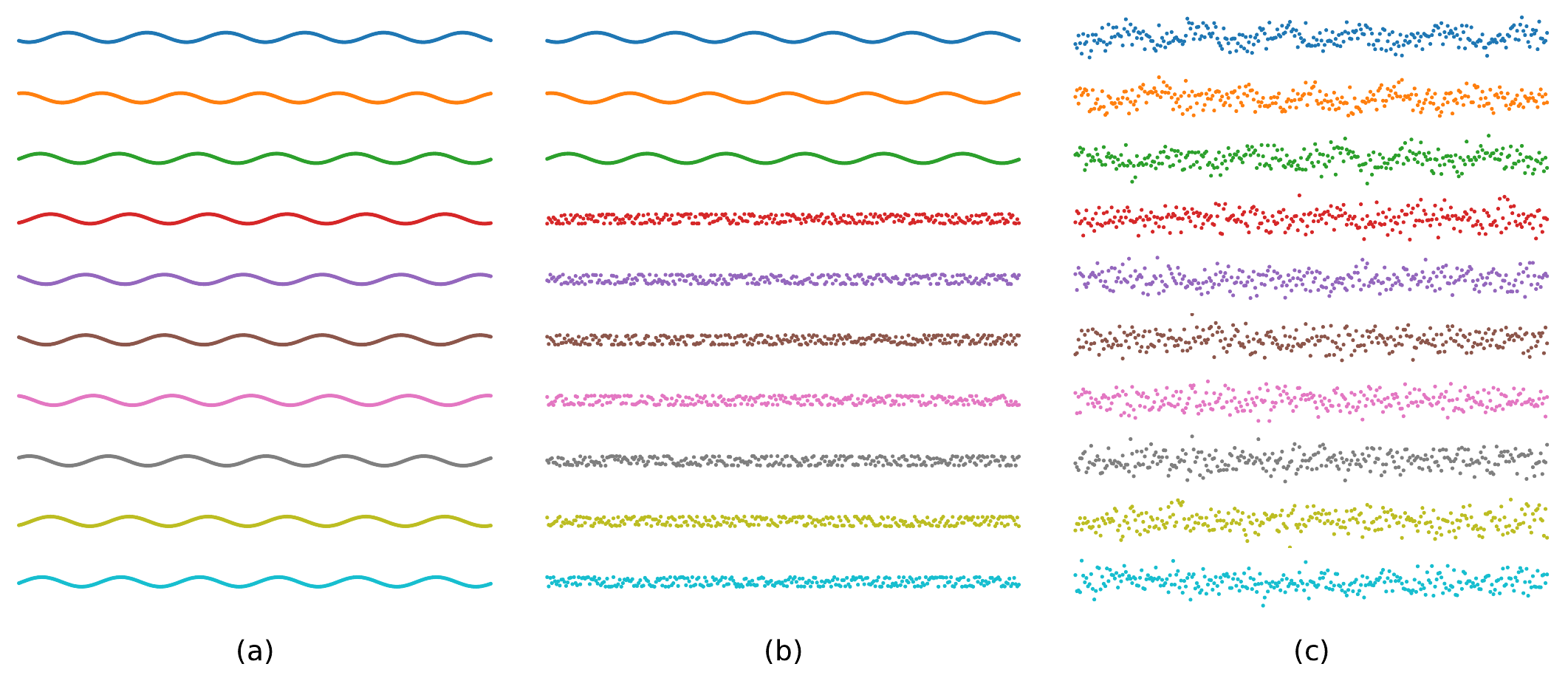}
\caption{The data set used in Example~\ref{ssec:periodicity}. 
Beginning with samples from ten sine waves (a), we randomly permute the values of seven of the signals (b) and add iid Gaussian noise with standard deviation $\tau = 1.5$ (c).} 
\label{fig:synth_data}
\end{figure}

Next, we compute the degree-$1$ persistent homology (Figure~\ref{fig:synth_pds}, left). 
Initializing at the barycenter, we optimize for maximum persistence using Algorithm~\ref{alg:stochastic-path} with initial learning rate $\eta_0\approx 8\times 10^{-6}$, decay factor $d=5\times 10^{-4}$, and gradient-noise scale $\sigma\approx 12.7$. As before, $\sigma$ is taken to be one tenth of the norm of the projected subgradient at the barycenter. With these choices the step length is $10^{-3}$ at the first iterate and decays to about half that value by the final iterate.
After performing 100 trials, we average the resulting paths through the geometric 9-simplex (Figure~\ref{fig:synth_pds} (left)). 
By \cref{thm:slln}, this sample mean gradient path converges to the mean gradient path.
Furthermore, by \cref{thm:clt}, the suitably normalized difference between the sample mean gradient path and the mean gradient path is normally distributed. 
The mean sample covariance of scores across these trials is approximately $2.4\times 10^{-6}$, and the covariance matrix is displayed in Figure~\ref{fig:synth_covariance}.

\begin{figure}[ht]
\centering
\includegraphics[width=.5\textwidth]{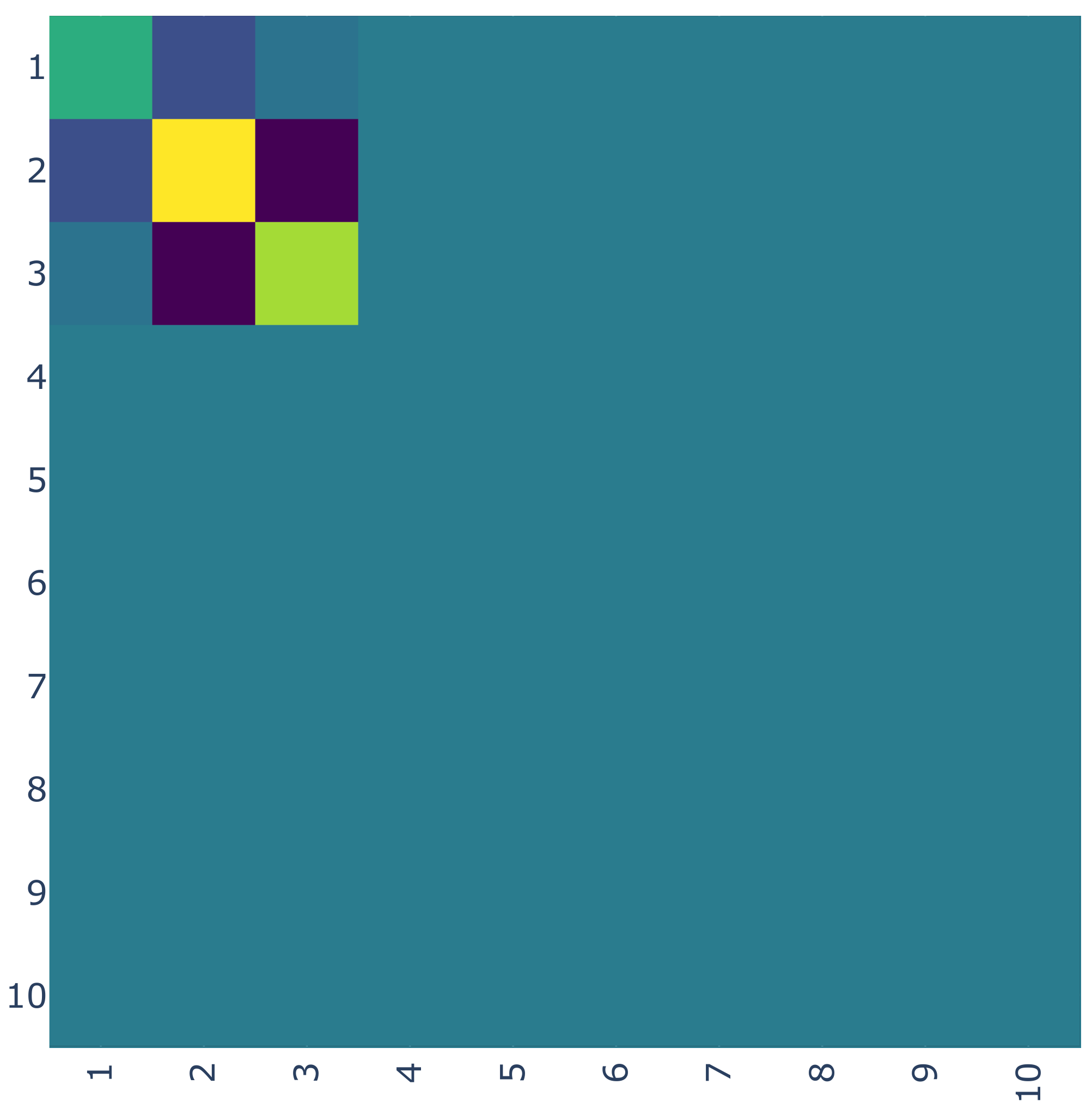}
\caption{The sample covariance matrix of  scores across 100 trials. }\label{fig:synth_covariance}
\end{figure}

The persistence diagram of the final convex-combination sliding window embedding is shown in Figure~\ref{fig:synth_pds}~(right), and the averaged path through the geometric $9$-simplex is visualized in Figure~\ref{fig:synth_path}.
The sample mean scores on the signals and standard deviations are shown in Table~\ref{table:scores}. 

\begin{table}[htbp]
\centering
\begin{tabular}{l|cccccccccc}
signal index & 1 & 2 & 3 & 4 & 5 & 6 & 7 & 8 & 9 & 10 \\
\hline
mean score & 0.317 & 0.296 & 0.387 & 0.000 & 0.000 & 0.000 & 0.000 & 0.000 & 0.000 & 0.000 \\
SD & 0.003 & 0.005 & 0.005 & 0.000 & 0.000 & 0.000 & 0.000 & 0.000 & 0.000 & 0.000 \\
\end{tabular}
\caption{Mean scores and standard deviations across 100  projected subgradient descents.}
\label{table:scores}
\end{table}

\begin{figure}[ht]
\centering
\includegraphics[width=.8\textwidth]{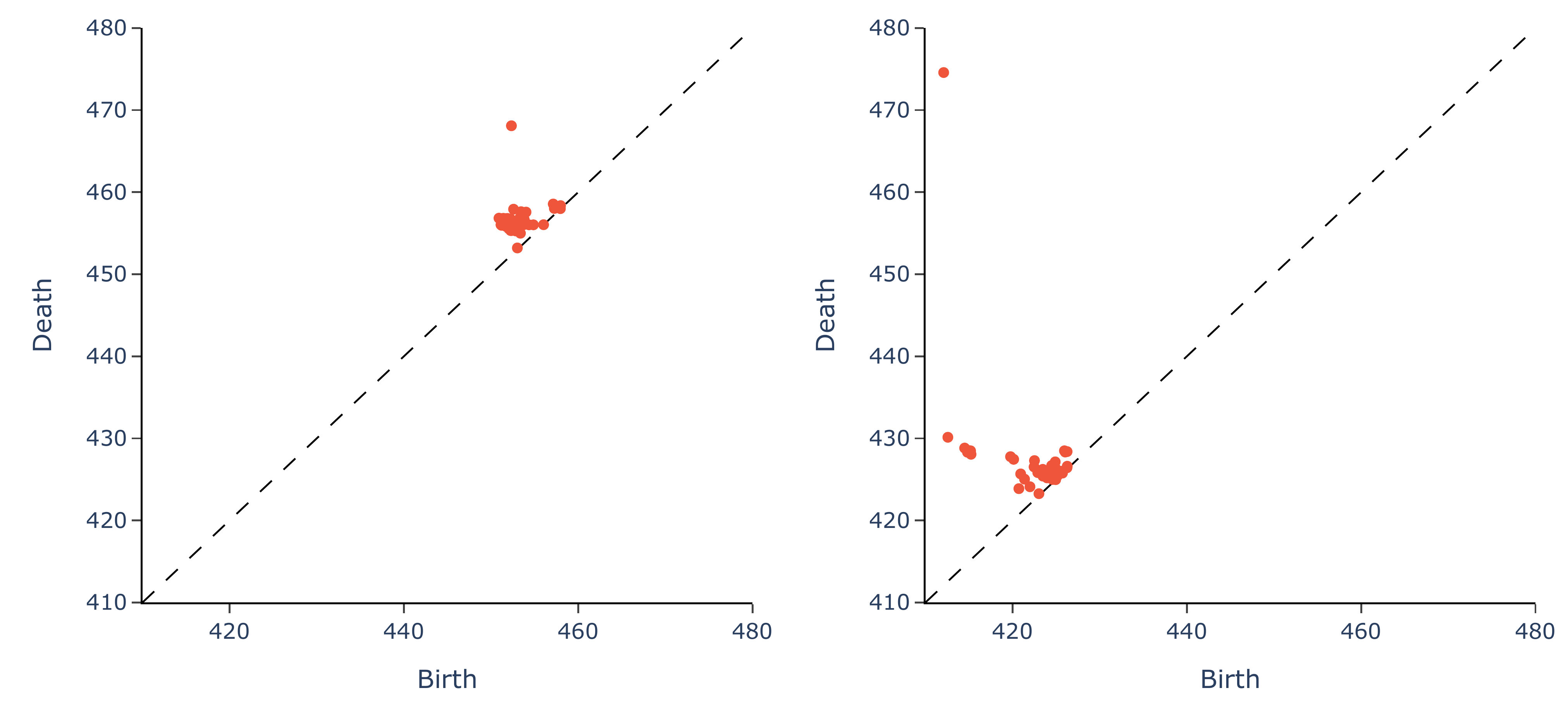}
\caption{The persistence diagram before optimization (left) and after optimization (right).}\label{fig:synth_pds}
\end{figure}

\begin{figure}[ht]
\centering
\includegraphics[width=0.9\textwidth]{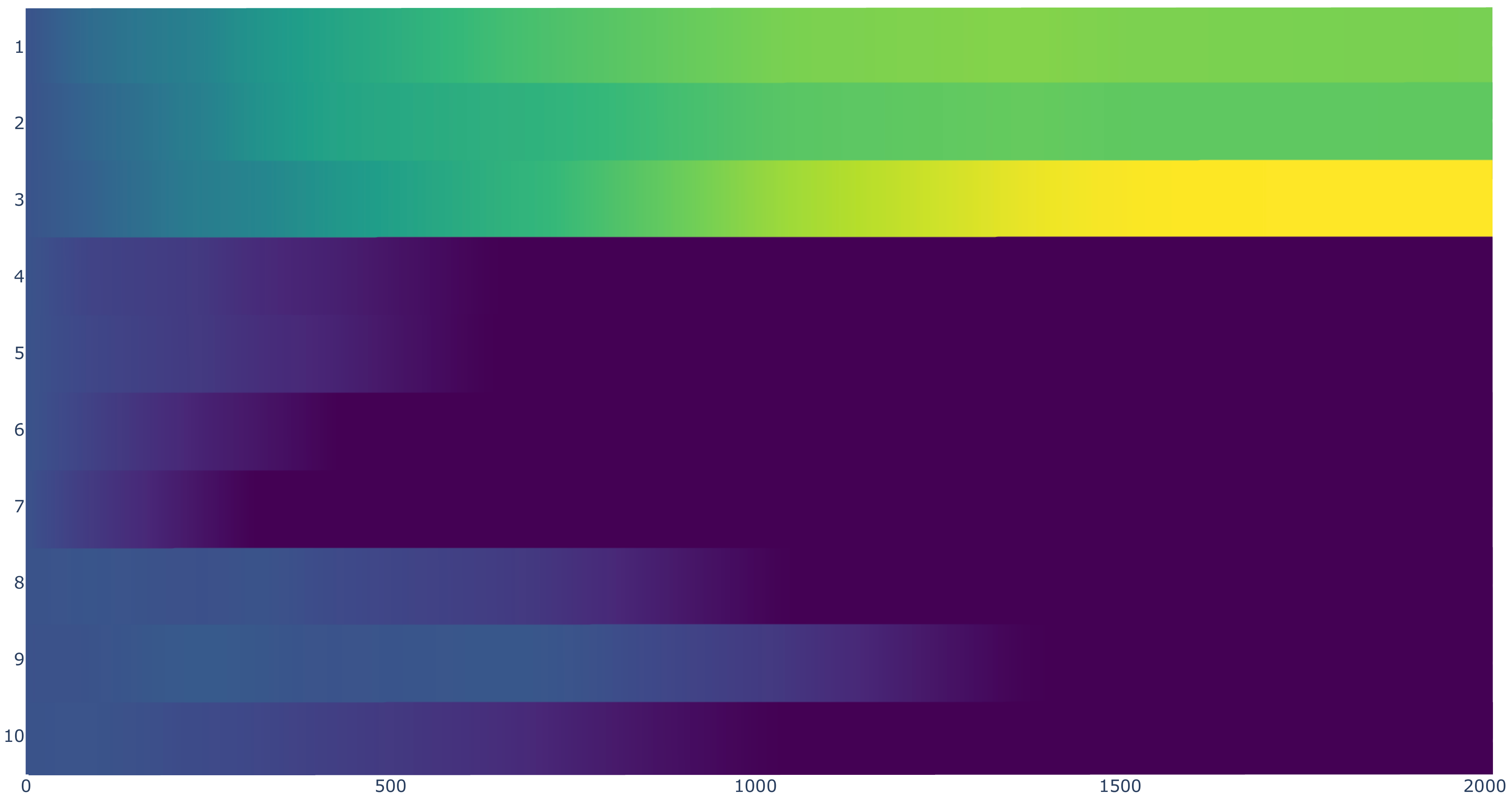}
\caption{A visualization of the average gradient path through the geometric $9$-simplex. 
The path starts at the barycenter (left edge) and ends (right edge) near the face spanned by the first three vertices.}\label{fig:synth_path}
\end{figure}

\subsection{Neuronal data}\label{ex:neuro} 

Finally, we apply our method to the scientific data that motivated our mathematical research.
The example in this section is based on neural activity within the model organism \emph{C.\thinspace elegans} obtained by Chaudhary et al.~\cite{chaudhary2021graphical}.
These organisms have been genetically modified to express GCaMP6s, a calcium-sensitive fluorescent protein that increases in fluorescence in response to calcium ion influx triggered by neuronal activation.
Because \emph{C.\thinspace elegans} is translucent, the intensity of this fluorescence provides an indirect measure of neuron-level activity.

For this dataset, the activity of $72$ neurons was recorded over a 278-second session at a resolution of one frame per second. 
During the recording, a microfluidic device was used to contain and apply \emph{E.\thinspace coli} OP50 to the nose tip of the animal to evoke chemosensory responses. 
This periodic stimulus (five seconds on, five seconds off) was applied starting at 100 seconds and ending at 180 seconds. 
Figure~\ref{fig:neuron_heatmap} shows a heatmap of the data set. 

\begin{figure}[ht]
\centering
\includegraphics[width=\textwidth]{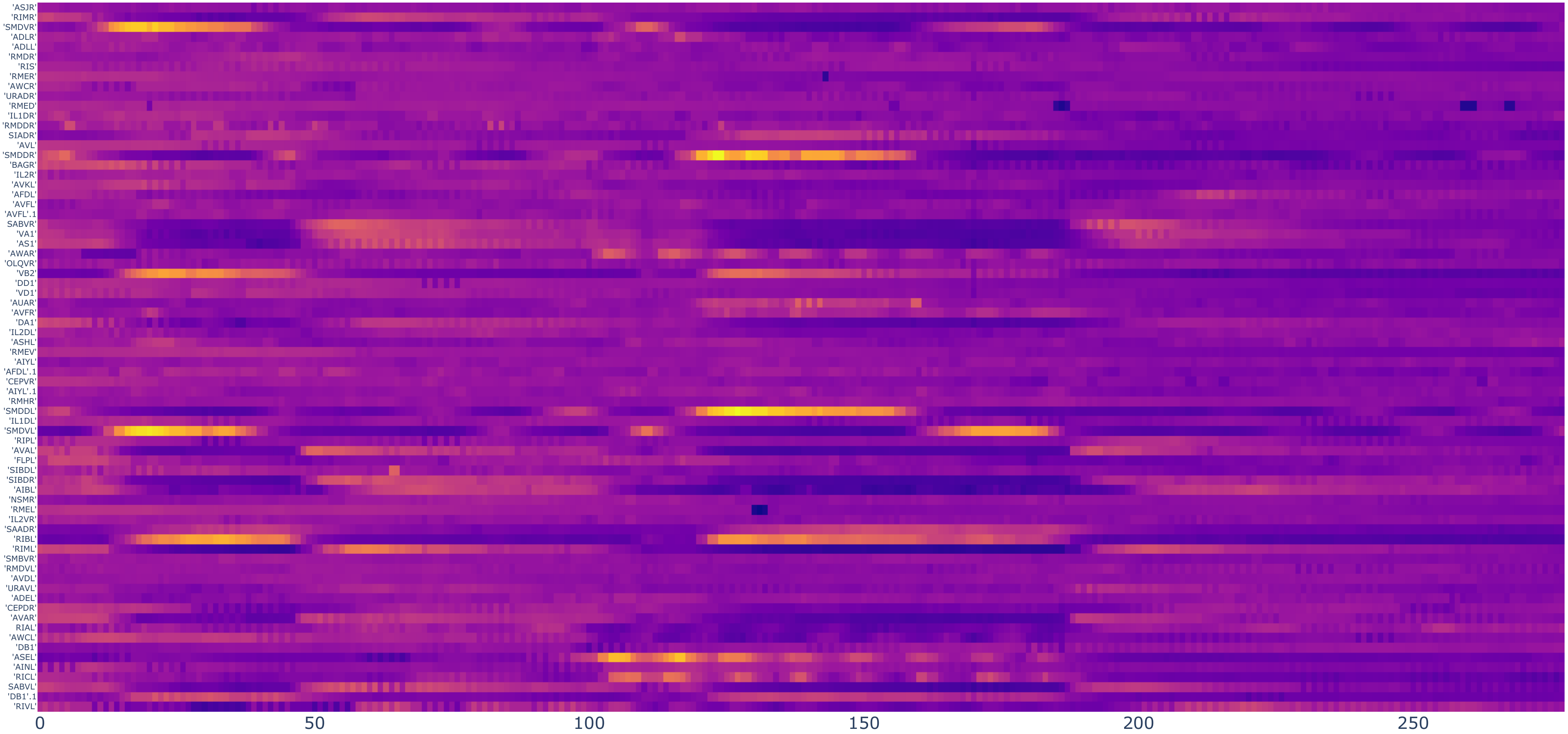}
\caption{Heatmap of the \emph{C.\thinspace elegans} neuronal data set. 
}\label{fig:neuron_heatmap}
\end{figure}

The sliding window embedding with window length $L=10$ yields a point cloud consisting of $269$ points in $\R^{720}$. 
This window length was chosen to align with the  $10$s period of the periodic stimulus. 
The augmented combinatorial simplex $K$ has dimension $268$. The optimization variable $v \in \abs{\Delta^{p-1}} = \abs{\Delta^{71}}$ represents a convex combination of the $72$ individual neuronal distance matrices.
The persistence diagram is displayed in Figure~\ref{fig:neuron_pds} (left). 
There are two large loops detected by persistence; indeed, the periodic stimulus applied throughout the middle third of the recording dramatically affects the trajectory of the orbit, forming the second loop. 
We therefore take $F$ to be the sum of the two largest lifetimes in the degree-$1$
persistence diagram, targeting the two dominant loops.

\begin{figure}[ht]
\centering
\includegraphics[width=.8\textwidth]{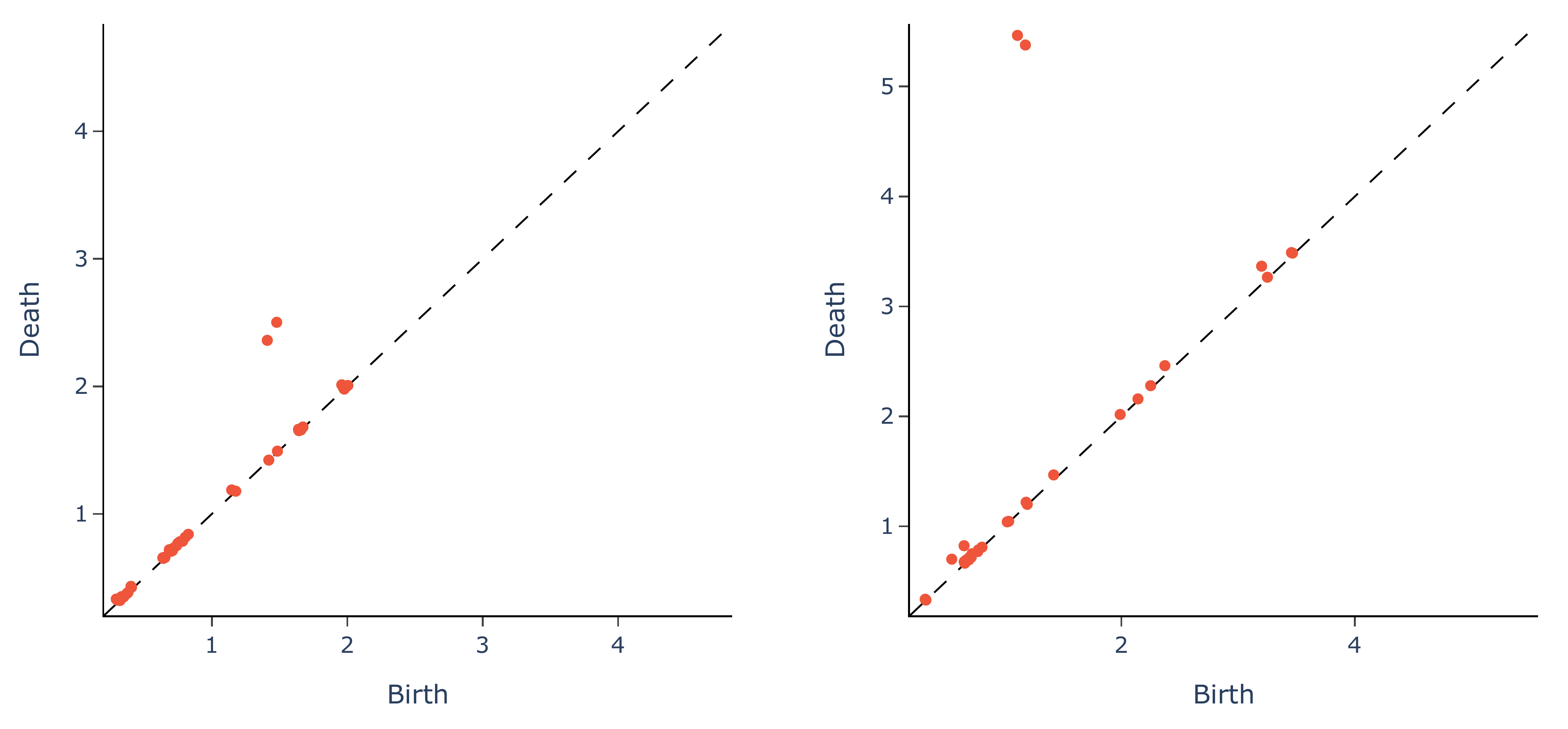}
\caption{The persistence diagram before optimization (left) and after optimization (right). 
}\label{fig:neuron_pds}
\end{figure}

We apply Algorithm~\ref{alg:stochastic-path} to optimize for the sum of the two longest lifetimes in the persistence diagram, using $1000$ steps with initial learning rate $\eta_0\approx 3\times 10^{-5}$, decay factor $d= 10^{-3}$, and gradient-noise scale $\sigma\approx 3.6$. Again $\sigma$ is chosen to be one tenth of the norm of the projected subgradient at the barycenter. With these choices the step length is $10^{-3}$ at the first iterate and decays to about half that value over the course of the descent. 
After performing $100$ trials, we average the resulting paths through the geometric $71$-simplex (Figure~\ref{fig:neuron_path}). 
The mean sample covariance is approximately $3.3\times 10^{-6}$.
The covariance matrix is displayed in Figure~\ref{fig:bio_covariance}.

\begin{figure}[ht]
\centering
\includegraphics[width=\textwidth]{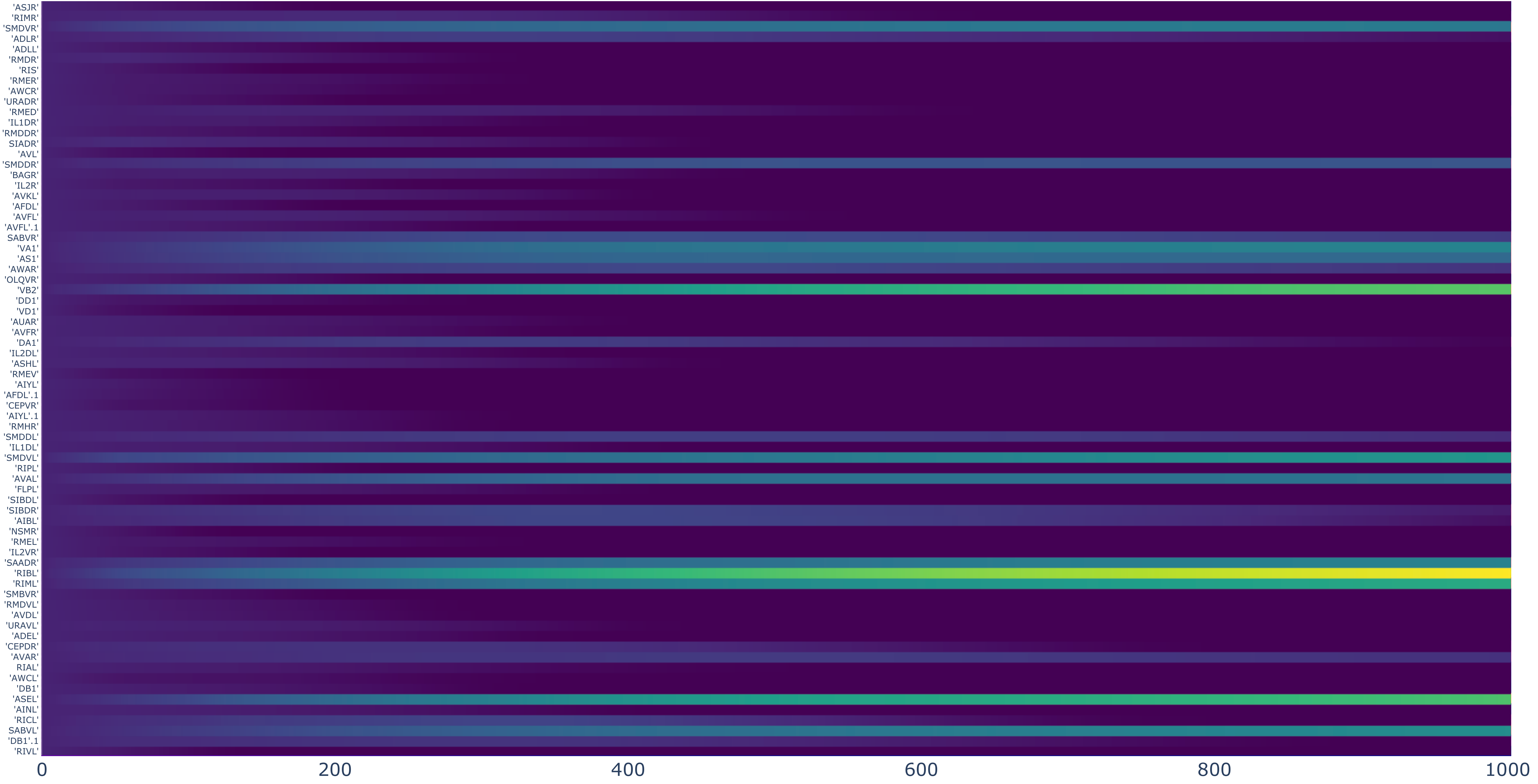}
\caption{A visualization of the sample mean gradient path through the geometric $71$-simplex. 
The path starts at the barycenter (left edge) and ends (right edge) close to the face given by the convex hull of the vertices corresponding to the neurons in \cref{fig:final_scores}.}\label{fig:neuron_path}
\end{figure}

\begin{figure}[ht]
\centering
\includegraphics[width=\textwidth]{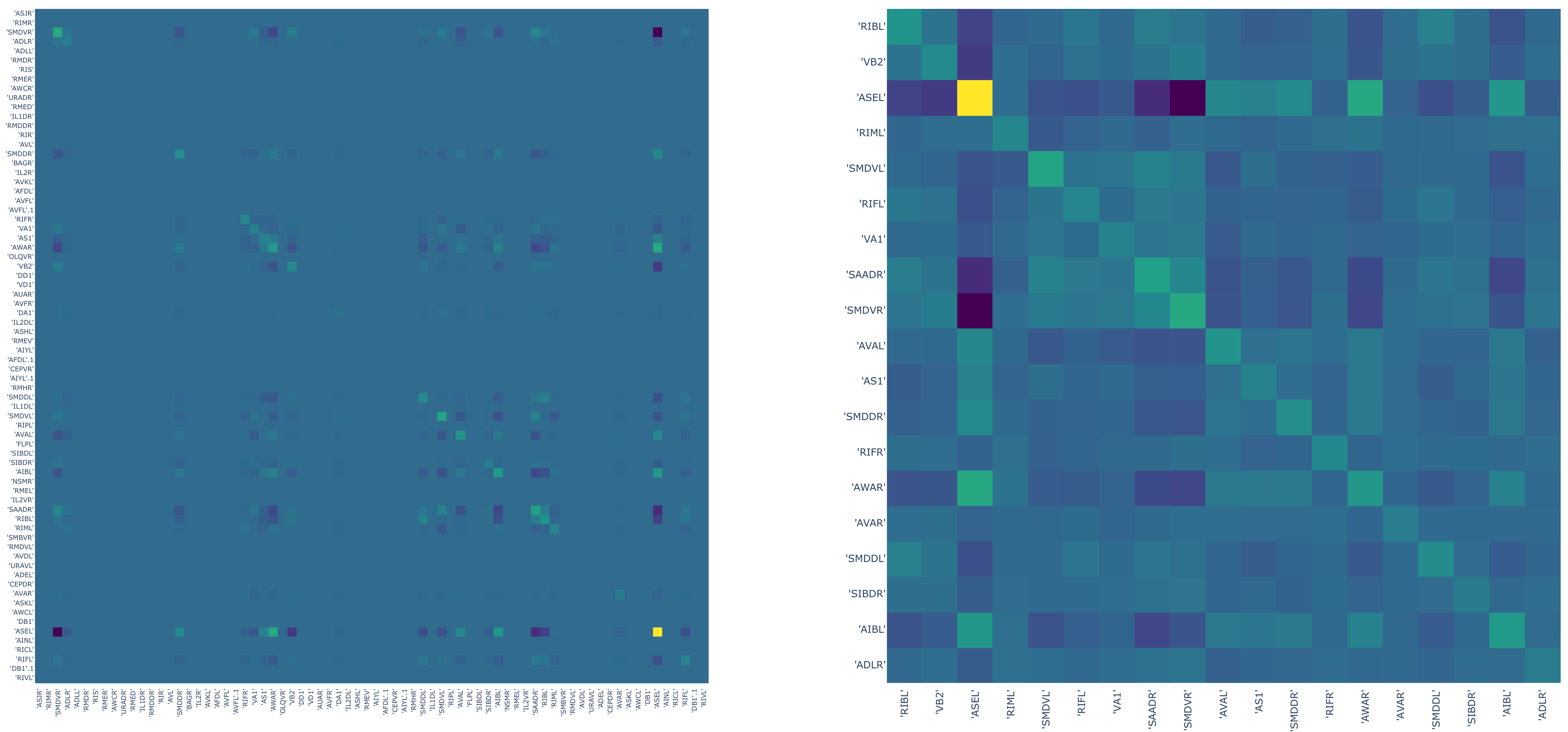}
\caption{The sample covariance matrix of scores across 100 trials (left), and the matrix restricted to neurons with a score greater than $10^{-3}$ (right). }\label{fig:bio_covariance}
\end{figure}

We find that $19$ of the $72$ neurons have a sample mean score greater than $10^{-3}$. 
These scores are plotted in Figure~\ref{fig:final_scores}. 
The most-highly scored neurons align with biologically significant components of the \emph{C.\thinspace elegans} nervous system, particularly those involved in sensory integration, motor control, and behavioral modulation. 
The top-scored neuron, RIBL, is a head interneuron known to integrate inputs from various sensory modalities and to contribute to locomotor regulation. 
Its prominent score suggests a central role in coordinating behavioral responses to environmental stimuli. 
Also highly ranked is ASEL, a chemosensory neuron specialized in detecting sodium ions, directly related to the periodic chemical stimulus used in the experiment. 
Several motor-related neurons score highly, including SMDVR, SMDVL, SMDDR, SMDDL, which are involved in head movement and posture, as well as VA1 and VB2, which belong to motor neuron classes responsible for backward and forward locomotion, respectively. 
The appearance of interneurons such as RIML, RIFL, and 
RIFR,
which participate in motor coordination and neuromodulatory pathways, indicates that the method selects not only sensory inputs but also downstream neurons involved in generating context-dependent behavior.

We also note that among the neurons with the highest score across the runs we observe many left/right pairs such as SMDVL/SMDVR, 
RIFL/RIFR,
SMDDL/SMDDR, and AVAL/AVAR. These neurons exhibit strongly correlated activity patterns, reflecting their anatomical symmetry and shared functional roles in motor control and sensory processing. Thus, this may be explained by the fact that their coordinated activity leads to similar contributions to the topological structure.

\begin{figure}[ht]
\centering
\includegraphics[width=.8\textwidth]{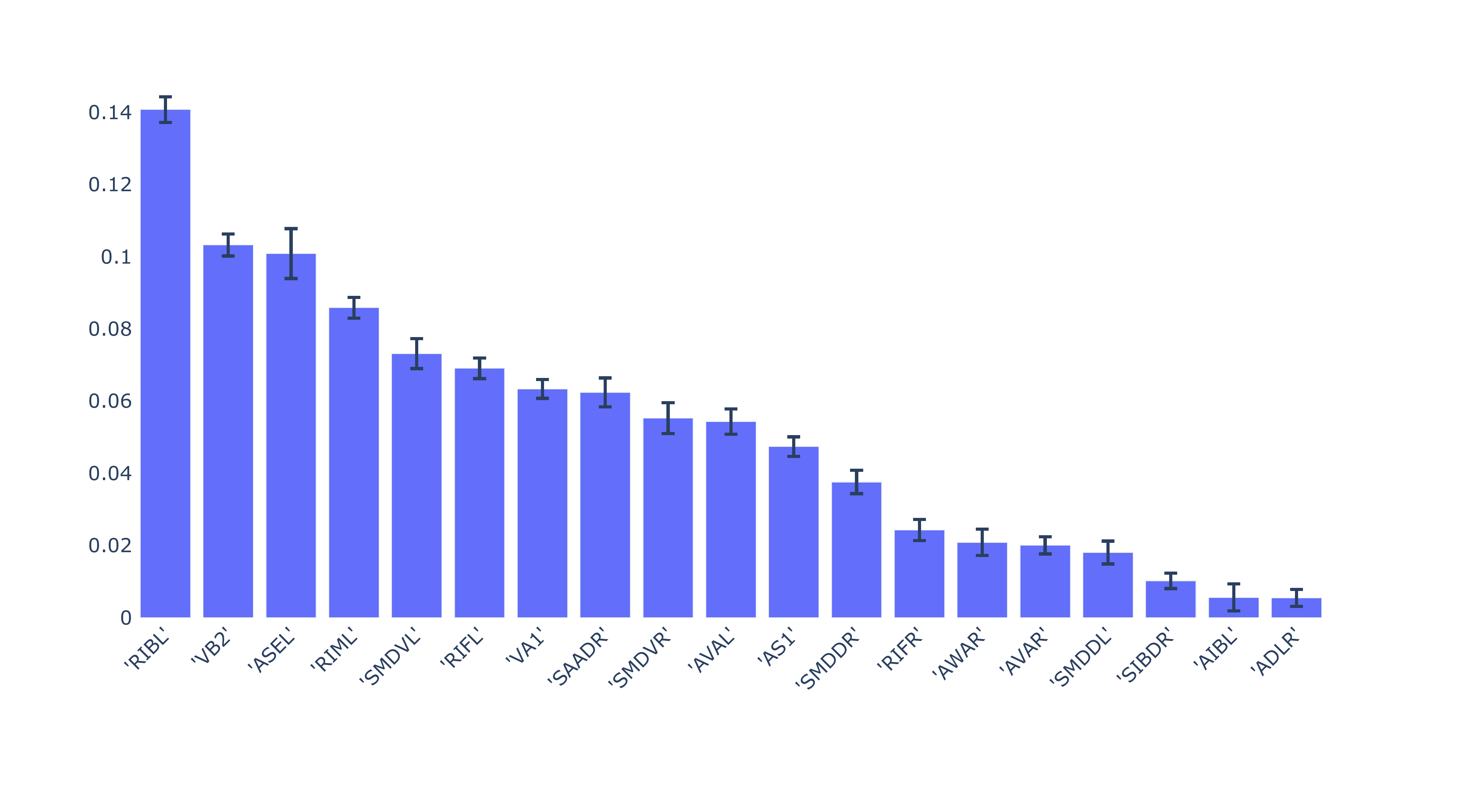}
\caption{The sample mean scores of the neurons with a score greater than $10^{-3}$, together with standard deviation error bars. 
}\label{fig:final_scores}
\end{figure}

To assess the robustness of our analysis with respect to changes in window length, we compared the sets of neurons with sample mean score greater than $10^{-3}$ across window lengths $L = 8, 9, 10, 11, 12$. See \cref{table:jaccard}.
The Jaccard index, defined to be the ratio of the cardinality of the intersection of two sets to the cardinality of their union, shows overlap across window lengths, suggesting that the results of this analysis are not particularly sensitive to modest variations in window length. 
\begin{table}[ht]
    \centering
    \begin{tabular}{l|ccccc}
Jaccard index & 8 & 9 & 10 & 11 & 12 \\
\hline
8  & 1.000 & 0.952 & 0.952 & 0.857 & 0.905 \\
9  &       & 1.000 & 1.000 & 0.900 & 0.950 \\
10 &       &       & 1.000 & 0.900 & 0.950 \\
11 &       &       &       & 1.000 & 0.947 \\
12 &       &       &       &       & 1.000 \\
\end{tabular}
    \caption{There is a high degree of overlap in the sets of neurons with sample mean score greater that $10^{-3}$ across window lengths $8$ through $12$.}
    \label{table:jaccard}
\end{table}

We note that while the neurons identified by our method appear active in the heatmap of Figure~\ref{fig:neuron_heatmap}, activity level and contribution to coordinated dynamics are distinct:\ the heatmap captures the amplitude and timing of individual neuronal responses, whereas the loop structure in Figure~\ref{fig:neuron_pds} reflects the geometry of the collective neural trajectory.

\section*{Discussion}

The two primary hyperparameters of our method are the window length $L$ and the objective function $F$.
When the period of the underlying dynamics is unknown, $L$ may be selected by computing persistence diagrams across a range of candidates; in our examples, $L$ was determined either by  identifying the one that maximized the quantity being optimized or by domain knowledge of the stimulus frequency.
The choice of $F$ is guided by the structure of the persistence diagram:\ in the \emph{C.~elegans} example, the presence of two prominent degree-$1$ features (Figure~\ref{fig:neuron_pds}) motivated our choice of $F$ as the sum of the two largest degree-$1$ lifetimes.

While our motivating examples focus on periodic and quasiperiodic dynamics, the method applies more broadly to any setting in which persistent homology provides a meaningful summary of the data, and any homological degree may be targeted.
A systematic comparison with other approaches to multivariate time series feature selection is an interesting direction for future work.
Optimizing the speed of our optimization would also be good direction for future work~\cite{MR4753555}.

\section*{Code availability}
Code to reproduce the numerical experiments in this paper is available at
\url{https://github.com/J-Bush/TFS}.

\section*{Acknowledgments}
We are grateful to Hang Lu and Sihoon Moon for many helpful discussions regarding the physiology of \emph{C.\thinspace elegans}. 
We are also grateful to Zev Woodstock for helpful discussions regarding optimization.
We thank the referees whose suggestions led to many improvements.
This research was partially supported by the National Science
Foundation (NSF) grant DMS-2324353 and by the Southeast Center
for Mathematics and Biology, an NSF-Simons Research Center for Mathematics of Complex
Biological Systems, under National Science Foundation Grant No.\ DMS-1764406 and Simons Foundation Grant No.~594594.

\bibliographystyle{plain}
\bibliography{references}

\end{document}